\documentclass[10pt]{amsart}
\usepackage[utf8]{inputenc}

\usepackage{amssymb,amsthm,amsmath}
\usepackage{mathrsfs}
\usepackage{enumerate}
\usepackage{graphicx,xcolor}
\usepackage[hidelinks]{hyperref}

\newcommand{\ls}{\leq}
\newcommand{\gr}{\geq}
\renewcommand{\geq}{\geqslant}
\renewcommand{\leq}{\leqslant}

\newcommand{\dd}{\mathrm{d}}
\newcommand{\E}{\mathbb{E}}

\newcommand{\R}{\mathbb{R}}
\newcommand{\C}{\mathbb{C}} 
\newcommand{\D}{\mathbb{D}}

\newcommand{\scal}[2]{\left\langle #1, #2 \right\rangle}
\newcommand{\red}{}
\newcommand{\purp}{\color{purple}}

\newcommand{\jj}{\mathfrak{j}}

\DeclareMathOperator{\vol}{vol}

\usepackage{xpatch}
\makeatletter
\def\thm@space@setup{%
  \thm@preskip=12pt plus 0pt minus 0pt
  \thm@postskip=0pt plus 0pt minus 0pt
}
\xpatchcmd{\proof}{6\p@\@plus6\p@\relax}{\z@skip}{}{}
\makeatother

\newtheorem{theorem}{Theorem}
\newtheorem{lemma}[theorem]{Lemma}
\newtheorem{corollary}[theorem]{Corollary}
\newtheorem{proposition}[theorem]{Proposition}
\theoremstyle{remark}
\newtheorem{remark}[theorem]{Remark}

\theoremstyle{definition}

\title{Haagerup's phase transition at polydisc slicing}

\author{Giorgos Chasapis}

\author{Salil Singh}

\author{Tomasz Tkocz}

\address{Giorgos Chasapis \\ University of Crete, Voutes Campus 70013, Heraklion, Crete, Greece.}

\email{gchasapis@uoc.gr}

\address{Salil Singh and Tomasz Tkocz\\ Carnegie Mellon University; Pittsburgh, PA 15213, USA.}

\email{\{salils,ttkocz\}@andrew.cmu.edu}

\thanks{GC supported by the Hellenic Foundation for Research and Innovation, Project HFRI-FM17-1733 and by University of Crete Grant 4725. TT's research supported in part by NSF grant DMS-1955175.}

\date{\today}

\begin{document}

\begin{abstract} 
We establish a sharp comparison inequality between the negative moments and the second moment of the magnitude of sums of independent random vectors uniform on three-dimensional Euclidean spheres. This provides a probabilistic extension of the Oleszkiewicz-Pe\l czy\'nski polydisc slicing result. The Haagerup-type phase transition occurs exactly when the $p$-norm recovers volume, in contrast to the real case. We also obtain partial results in higher dimensions.
\end{abstract}

\maketitle

\bigskip

\begin{footnotesize}
\noindent {\em 2010 Mathematics Subject Classification.} Primary 60E15; Secondary 52A20, 33C10.

\noindent {\em Key words. polydisc slicing, Bessel function, negative moments, Khinchin inequality, sharp moment comparison, sums of independent random vectors, uniform spherically symmetric random vectors.} 
\end{footnotesize}

\bigskip

\section{Introduction}

Khinchin-type inequalities concern estimates on $L_p$ norms of (weighted) sums of independent random variables, typically involving a norm which is easily understood (or explicit in given parameters) such as the $L_2$ norm. They can be traced back to Khinchin's work \cite{Kh} on the law of the iterated logarithm, where he established such bounds for Rademacher random variables (random signs). Beyond their original use, most notably, such inequalities have played an important role in Banach space theory (in connection with topics such as unconditional convergence or type and cotype), see \cite{Die, Ver2, LT, Woj}. Considerable work has been devoted to the pursuit of sharp constants in Khinichin-type inequalities, see for instance \cite{AH, BN, ENT1, ENT2, Haa, HNT, KLO, LO-best, LO, Mo, NO, NP, New, koles, koles-b, PS, Sz, Whi, Y}, in particular for sums of random vectors uniform on Euclidean spheres \cite{BC, CGT, CKT, Ko, KK} (as a natural generalisation of Rademacher and Steinhaus random variables, intimately related to uniform convergence in real and complex Banach spaces, respectively). This paper continues that line of research. 

Throughout, $|\cdot|$ denotes the standard Euclidean norm on $\R^d$, inherited from the standard inner product $\scal{\cdot}{\cdot}$. For a random vector $X$ in $\R^d$ and a real parameter $p$, we write $\|X\|_p = (\E|X|^p)^{1/p}$ for the $L_p$-norm ($p$-th moment) of the magnitude of $X$ (whenever the expectation exists, with $p=0$ understood as usual as $\|X\|_0 = e^{\E \log|X|}$, arising from taking the limit as $p \to 0$).

Let $\xi_1, \xi_2, \dots$ be independent random vectors, each uniform on the unit Euclidean sphere $S^{d-1}$ in $\R^d$. In particular, when $d=1$, these are Rademacher random variables, that is symmetric random signs in $\R$, whereas when $d=2$, they are often referred to as Steinhaus random variables (especially when $\R^2$ is treated as $\C$). For $q > -(d-1)$, let $c_{d}(q)$ be the best positive constant such that the following Khinchin-type inequality holds: for every $n \geq 1$ and real scalars $a_1, \dots, a_n$, we have
\begin{equation}\label{eq:khin}
\left\|\sum_{k=1}^n a_k\xi_k\right\|_q \geq c_d(q)\left\|\sum_{k=1}^n a_k\xi_k\right\|_2.
\end{equation}
In other words, thanks to homogeneity, $c(q)$ is the infimal value of $\left\|\sum_{k=1}^n a_k\xi_k\right\|_q$ over all $n \geq 1$ and $a_1, \dots, a_n \in \R$ with $\sum a_k^2 = 1$. {\red We stress that when $d\gr 1$ and $q > -(d-1)$, this $L_q$ norm exists  regardless of the coefficients, e.g. seen by noting that then $\E|\xi_1+x|^q = \E(|x|^2+2\scal{x}{\xi_1}+1)^{q/2}$ is finite for every $x \in \R^d$, using that $\scal{x}{\xi_1}$ has density proportional to $(1-(u/|x|)^2)^{\frac{d-3}{2}}$ on $-|x| \leq u \leq |x|$ (of course, for a given sequence of coefficients $a_j$, the range of $q$ may be larger, for instance when $d=1$, it is all $q \in \R$ as long as $\sum_{j=1}^n \pm a_j$ never vanishes)}.

Plainly, $c_{d}(q) = 1$ for $q \geq 2$ (by the monotonicity of $p \mapsto \|\cdot\|_p$). When $q \geq 2$, the reverse inequality to \eqref{eq:khin} is nontrivial and interesting, but we do not discuss it here at all, referring instead to, for instance \cite{BC, HT, NO} for a comprehensive account of known as well as recent results.

From now on we consider $-(d-1)< q < 2$. 
We define two  constants arising from two particular choices of weights in \eqref{eq:khin}: $a_1 = a_2 = \frac{1}{\sqrt{2}}$ with $n=2$ and $a_1 = \dots = a_n = \frac{1}{\sqrt{n}}$ with $n \to \infty$, 
\begin{align}\label{eq:const-2}
c_{d,2}(q) &= \left\|\frac{\xi_1+\xi_2}{\sqrt{2}}\right\|_q = \frac{1}{\sqrt{2}}\left(\frac{\Gamma\left(\frac{d}{2}\right)\Gamma(d+q-1)}{\Gamma\left(\frac{d+q}{2}\right)\Gamma\left(d+\frac{q}{2}-1\right)}\right)^{1/q}, \\\label{eq:const-inf}
c_{d,\infty}(q) &= \lim_{n\to\infty} \left\|\frac{\xi_1+\dots+\xi_n}{\sqrt{n}}\right\|_q = \left\|\frac{Z}{\sqrt{d}}\right\|_q = \sqrt{\frac{2}{d}}\left(\frac{\Gamma\left(\frac{d+q}{2}\right)}{\Gamma\left(\frac{d}{2}\right)}\right)^{1/q},
\end{align}
where $Z$ is a standard Gaussian random vector in $\R^d$ (emerging by the central limit theorem). The expression for $c_{d,2}(q)$ will be justified later (see Corollary \ref{cor:const-2}), whereas the expression for $c_{d,\infty}(q)$ follows by a simple integration in polar coordinates. Note that 
\begin{equation}\label{eq:c-low}
c_d(q) {\red \leq} \min\{c_{d,2}(q), c_{d,\infty}(q)\}.
\end{equation}
It can be checked that in fact
\begin{equation}\label{eq:phase.transition}
\min\{c_{d,2}(q), c_{d,\infty}(q)\} = \begin{cases} c_{d,2}(q), & -(d-1) < q \leq q_d^*, \\ c_{d,\infty}(q), & q_d^* \leq q \leq 2,  \end{cases}
\end{equation}
where $q_d^*$ is the unique solution of the equation $c_{d,2}(q) = c_{d,\infty}(q)$ in $(-(d-1),2)$. We have included a sketch of the proof of this fact in the \hyperref[app:pht]{appendix}. In Table \ref{tab:q*} below we list some numerical values of $q_d^*$. We are grateful to Hermann K\"onig for sharing his notes on these topics, \cite{Ko-priv}.

\subsection{Known results}
The pursuit of the value of $c_d(q)$ has a rich history which can be summarised in one simple statement that in all known cases {\red the trivial bound \eqref{eq:c-low} is tight}. Of course, the history begins with the one dimensional case of Rademacher random variables. In his study \cite{Lit} on bilinear forms, Littlewood conjectured that $c_1(1) = c_{1,2}(1) = \frac{1}{\sqrt{2}}$, which was confirmed by Szarek in \cite{Sz} (see also \cite{LO-best} and \cite{Tom}). Haagerup's pivotal work \cite{Haa} addressed the entire range $0 < q < 2$, showing the following phase transition in the behaviour of $c_1(q)$:
\[
c_1(q) = \begin{cases} c_{1,2}(q), & 0 < q \leq q_1^*, \\ c_{1,\infty}(q), & q_1^* \leq q < 2,  \end{cases}
\]
where $q_1^* = 1.84..$ is the unique solution of the equation $c_{1,2}(q) = c_{1,\infty}(q)$ in $(0,2)$; in particular, when $d=1$, we have equality in \eqref{eq:c-low}. We also refer to Nazarov and Podkorytov's paper \cite{NP} which offered great simplifications. Haagerup devised a very efficient argument, crucially relying on Fourier-analytic formulae for $L_p$-norms, which together with \cite{NP} paved the path for many further results. 

That a similar behaviour occurs in the case $d=2$ (Steinhaus variables) was conjectured by Haagerup, later confirmed by K\"onig in \cite{Ko}: when $d=2$, $0 \leq q < 2$, we have equality in \eqref{eq:c-low} and the phase transition occurs now at $q_2^* = 0.47..$. The range $1 \leq q < 2$ was in fact earlier dealt with by K\"onig and Kwapie\'n in \cite{KK} (with $q=1$ handled even earlier by Sawa in \cite{Saw}), whereas $-1  <q < 0$  (to the best of our knowledge) appears to be left open, with a natural conjecture that $c(q) = c_{2,2}(q)$.

For the case $d=3$: Lata\l a and Oleszkiewicz showed in \cite{LO} that $c_{3}(q) = c_{3,\infty}(q)$ for $1 \leq q < 2$ which was extended to $0 < q < 1$ in our joint work \cite{CGT} with Gurushankar (see Proposition \ref{prop:ball-sphere} below for a connection to uniform distribution on intervals). The phase transition occurs in the range $-1 < q < 0$ at $q_3^* = -0.79..$, as established in our joint work \cite{CKT} with K\"onig, so when $d=3$ and  $-1 < q  < 2$, \eqref{eq:c-low} holds with equality. Again, $-2 < q < -1$ appears to be open with a natural conjecture that $c(q) = c_{3,2}(q)$. 

In higher dimensions $d \geq 4$, there are precise Schur-convexity results available for positive moments due to Baerstein II and Culverhouse from \cite{BC} and, independently K\"onig and Kwapie\'n from \cite{KK}: when $0 \leq q < 2$, it follows in particular that $c_{d}(q) = c_{d,\infty}(q)$. However, nothing seems to be known about the value of $c_{d}(q)$ for negative $q$, except it being (nontrivially) finite, as shown by Gorin and Favorov in \cite{GF} (in a much more general setting).  This paper partially fills out this gap.

\subsection{Our contribution}
Our first result concerns the best constant $c_d(q)$ in the inequality \eqref{eq:khin} when $q>-(d-4)$. It turns out that this is a consequence of a Schur-concavity type statement that follows directly from the main result of \cite{BC} (see Theorem \ref{thm:high-dim-Schur} below).

\begin{theorem}\label{thm:high-dim}
For every $d\gr 5$ and $-(d-4)\ls q<0$, we have $c_d(q)=c_{d,\infty}(q)$.
\end{theorem}

{\red Note that the restriction $-(d-4)\ls q<0$ already makes the statement of Theorem~\ref{thm:high-dim} meaningful only for dimensions $d \geq 5$}. Our second result covers the entire range $-3 < q < 0$ for dimension $d = 4$, which exhibits Haagerup's phase transition at exactly $q_4^* = -2$ (see also Table \ref{tab:q*} for other values of $q_d^*$ and a summary of known results and open questions).

\begin{theorem}\label{thm:d=4}
For $-3 < q < 0$, we have
\[
c_{4}(q) = \begin{cases} c_{4,2}(q), & -3 < q \leq -2, \\ c_{4,\infty}(q), & -2 \leq q < 0. \end{cases}
\]
\end{theorem}

\begin{table}[!hb]
\begin{center}
\caption{Numerical values of $q_d^*$ (see \eqref{eq:q*-d-large} for its asymptotics), known results and open questions about the best constant in Khinchin inequality \eqref{eq:khin}.}
\label{tab:q*}
\hspace*{-4em}\begin{tabular}{r|l|l|p{1.5cm}|c}
$d$ & $q_d^*$  & Range where $c(q)$ known &  Phase  transition & Left open \\\hline
$1$ & $1.82..$ & $0 < q < 2$ (\cite{Haa}) & \cite{Haa} & $-$ \\
$2$ & $0.47..$ & $0 < q < 2$ (\cite{BC, Ko, KK}) & \cite{Ko} & $-1 < q < 0$ \\
$3$ & $-0.79..$ & $-1 < q < 2$ (\cite{CGT, CKT, LO}) & \cite{CKT} & $-2 < q < -1$ \\
$4$ & $-2$ & $-3 < q < 2$ (Thm. \ref{thm:d=4}) & Thm. \ref{thm:d=4} & $-$ \\\hline
$5$ & $-3.16..$ & $-1 < q < 2$ (\cite{BC, KK}, Thm. \ref{thm:high-dim}) & \hspace*{15pt} ? & $-4 < q < -1$ \\
$\vdots$ & & & & \\
$d$ & $-(d-1)+o(1)$ & $-(d-4) < q < 2$ (\cite{BC, KK}, Thm. \ref{thm:high-dim})  & \hspace*{15pt}  ? & $-(d-1) < q < -(d-4)$
\end{tabular}
\end{center}
\end{table}

\subsection{Relation to volume}
It can perhaps be traced back to Kalton and Koldobsky's paper \cite{KalKol} that the volume of hyperplane sections of convex bodies can be expressed in terms of negative moments (of linear forms in vectors uniform on the body). Brzezinski's work \cite{Brz} makes the same connection for sections of products of Euclidean balls by block subspaces and our recent work with Nayar \cite{CNT} explores this further. In particular, as \cite{CKT} extends Ball's cube slicing result from \cite{Ball} (in the form of sharp Khinchin inequality \eqref{eq:khin} when $d=3$), Theorem \ref{thm:d=4} can be viewed as a probabilistic extension of Oleszkiewicz and Pe\l czy\'nski's polydisc slicing from \cite{OP}.  In fact, this connection was the main motivation of this work. It is very intriguing that the phase transition occurs exactly at $q=-2$ which is when \eqref{eq:khin} recovers the result for volume from \cite{OP}.

More specifically, let $\D = \{z \in \C, \ |z| < 1\}$ be the unit disc in the complex plane. Oleszkiewicz and Pe\l czy\'nski in \cite{OP} proved the following sharp inequality about extremal-volume (complex) hyperplane sections of the polydics $\D^n$ in $\C^n$: for every (complex) codimension $1$ subspace $H$ in $\C^n$, we have
\begin{align}\label{eq:OP-up}
\vol_{2n-2}(\D^n \cap H) &\leq \vol_{2n-2}(\D^n \cap (1,1,0,\dots,0)^\perp),\\\label{eq:OP-low}
\vol_{2n-2}(\D^n \cap H) &\geq \vol_{2n-2}(\D^n \cap (1,0,\dots,0)^\perp).
\end{align}
Here $a^\perp = \{z \in \C^n, \scal{a}{z} = 0\}$ is the (codimension $1$) hyperplane orthogonal to a vector $a$ in $\C^n$ and $\scal{\cdot}{\cdot}$ is the standard inner product in $\C^n$. 
If we let $U_1, \dots, U_n$ be independent random vectors, each uniform on $\D$ and let $a = (a_1, \dots, a_n)$ be a unit vector in $\C^n$, then
\[
\vol_{2n-2}(\D^n \cap a^\perp) =  \frac{\pi^{n-1}}{2}\lim_{p \to 2-} (2-p)\E\left|\sum_{k=1}^n a_kU_k\right|^{-p}
\]
(such formulae hold for arbitrary origin-symmetric convex sets, and this one follows immediately from Corollary 11 in \cite{CNT}). 
Moreover, the moments of sums of vectors uniform on balls are proportional to sums of vectors uniform on spheres (in a slightly higher dimension).

\begin{proposition}[\cite{BC}, \cite{KK}]\label{prop:ball-sphere}
Let $d \geq 3$ and let $\xi_1, \xi_2, \dots$ be independent random vectors uniform on the unit Euclidean sphere $S^{d-1}$ in $\R^d$ and let $U_1, U_2, \dots$ be independent random vectors uniform on the unit Euclidean ball $\mathbb{B}^{d-2}$ in $\R^{d-2}$. For every $q > -(d-2)$, $n \geq 1$ and scalars $a_1, \dots, a_n$, we have
\[
\E\left|\sum_{k=1}^n a_kU_k\right|^q = \frac{d-2}{d-2+q}\E\left|\sum_{k=1}^n a_k\xi_k\right|^q
\]
\end{proposition}

This identity can be seen in a number of ways, but essentially it follows from the folklore result that if a random vector $\xi = (\xi_1, \dots, \xi_d)$ is uniform on $S^{d-1}$, then its projection $(\xi_1, \dots, \xi_{d-2})$ onto $\R^{d-2}$ is uniform on $\mathbb{B}^{d-2}$. Specialised to $d=4$ and combined with the previous formula, it yields
\[
\vol_{2n-2}(\D^n \cap a^\perp) = \pi^{n-1}\E\left|\sum_{k=1}^n a_k\xi_k\right|^{-2}
\]
(see also \cite{KKol} and \cite{KRud} for generalisations to noncentral sections).
Thus, the upper bound \eqref{eq:OP-up} is Theorem \ref{thm:d=4} at $q=-2$, that is {\red $c_4(-2) = c_{4,2}(-2)$}. Incidentally, the lower bound \eqref{eq:OP-low} follows immediately from Jensen's inequality (see, e.g. \cite{Brz}, or \cite{KKol}, as well as \cite{CNT} for a stability result).

The sequel is devoted to proofs. First we provide some background and give a brief summary. Then we move to the proof of Theorem \ref{thm:high-dim} (which is very short) and the rest is occupied with the proof of Theorem \ref{thm:d=4}.

\subsection*{Acknowledgements.} 
We should very much like to thank Hermann K\"onig for the encouraging and helpful correspondence. {\red We are also immensely indebted to an anonymous referee for their very careful reading of the manuscript and numerous invaluable suggestions.}

\section{Proofs of the main results}

\subsection{Some background and outline}

Theorem \ref{thm:high-dim} will follow easily from the main result of \cite{BC}. As for positive moments, the point is that the range $-(d-4) < q < 0$ still warrants \emph{enough} convexity of the underlying moment functional, specifically the function $|x|^{q}$ (in fact, its $C^\infty$ regularisation/approximation) is bisubharmonic.

When $d=4$, as in Theorem \ref{thm:d=4}, this range is empty, Schur convexity/concavity does not hold, and more subtle arguments are needed. We will employ a Fourier-analytic approach (pioneered by Haagerup for random signs in \cite{Haa}). On its own however, this does not dispense of all cases. We extend an inductive argument of Nazarov and Podkorytov from \cite{NP} to our multidimensional setting and \emph{all} negative moments (building on \cite{CKT} with new ideas needed to go beyond the $-1$st moment).
The Fourier-analytic approach relies on the following integral representation of Gorin and Favorov for negative moments.

\begin{lemma}[Lemma 3 in \cite{GF}]\label{lm:formula-mom}
For a random vector $X$ in $\R^d$ and $0 < p < d$, we have
\begin{equation}\label{eq:formula-mom}
\E|X|^{-p} = K_{p,d}\int_{\R^d}\Big(\E e^{i\scal{t}{X}}\Big)|t|^{p-d} \dd t,
\end{equation}
provided that the right hand side integral exists, where
\[
K_{p,d} = 2^{-p}\pi^{-d/2}\frac{\Gamma\left(\frac{d-p}{2}\right)}{\Gamma\left(\frac{p}{2}\right)}.
\]
\end{lemma}

Of course, the Fourier transform (the characteristic function) goes hand in hand with independence. The trade-off is that when applied to sums of independent random vectors uniform on spheres, highly-oscillating integrands appear, more precisely, the Bessel functions. To recall, for integral $k \geq 0$ and real $x$, we use the notation
\[
(x)_k = \frac{\Gamma(x+k)}{\Gamma(x)} = x(x+1)\dots(x+k-1)
\]
for the rising factorial (Pochhammer symbol). 
Throughout,
\[
J_\nu(t) = \sum_{k=0}^\infty \frac{(-1)^k}{k!\Gamma(k+\nu+1)}\left(\frac{t}{2}\right)^{2k+\nu}
\]
is the Bessel function of the first kind with parameter $\nu > 0$. We also introduce the function
\begin{equation}\label{eq:defjj}
\jj_\nu(t) = 2^\nu\Gamma(\nu+1)t^{-\nu}J_\nu(t) = \sum_{k=0}^\infty \frac{(-1)^k}{k!(\nu+1)_k}\left(\frac{t}{2}\right)^{2k}.
\end{equation}
Its importance stems from the fact that for a random vector $\xi$ uniform on the unit Euclidean sphere $S^{d-1}$ in $\R^d$ and a vector $v$ in $\R^d$, we have
\begin{equation}\label{eq:jj}
\E e^{i\scal{v}{\xi}} = \jj_{d/2-1}(|v|)
\end{equation}
(see, e.g. the proof of Proposition 10 in \cite{KK}). This combined with Lemma \ref{lm:formula-mom} gives the following corollary.

\begin{corollary}\label{cor:neg-mom-rot}
For independent, rotationally invariant random vectors $X_1, \dots, X_n$ in $\R^d$ and $0 < p < d$, we have
\begin{equation}\label{eq:neg-mom-rot}
\E\left|\sum_{k=1}^n X_k\right|^{-p} = \kappa_{p,d}\int_{0}^\infty \prod_{k=1}^n\Big(\E\ \jj_{d/2-1}(t|X_k|)\Big) t^{p-1} \dd t,
\end{equation}
provided that the right hand side integral exists, where
\[
\kappa_{p,d} = 2^{1-p}\frac{\Gamma\left(\frac{d-p}{2}\right)}{\Gamma\left(\frac{d}{2}\right)\Gamma\left(\frac{p}{2}\right)}.
\]
\end{corollary}
\begin{proof}
Let $\xi_1, \dots, \xi_n$ be independent random vectors, each uniform on the unit Euclidean sphere $S^{d-1}$, chosen independently of the $X_k$. Then $X_k$ has the same distribution as $|X_k|\xi_k$ and \eqref{eq:formula-mom} together with \eqref{eq:jj} and integration in polar coordinates give
\begin{align*}
\E\left|\sum_{k=1}^n X_k\right|^{-p} &= K_{p,d}\int_{\R^d} \left(\prod_{k=1}^n \E e^{i\scal{t}{|X_k|\xi_k}} \right)|t|^{p-d} \dd t \\
&= K_{p,d}\int_{\R^d} \left(\prod_{k=1}^n \E\ \jj_{d/2-1}(|t||X_k|) \right)|t|^{p-d} \dd t \\
&= K_{p,d}|S^{d-1}|\int_0^\infty  \left(\prod_{k=1}^n \E\ \jj_{d/2-1}(t|X_k|) \right)t^{p-1} \dd t,
\end{align*}
where $|S^{d-1}| = \frac{2\pi^{d/2}}{\Gamma\left(\frac{d}{2}\right)}$ is the $(d-1)$-dimensional volume of the unit sphere in $\R^d$.
\end{proof}

\subsection{Proof of Theorem \ref{thm:high-dim}}

Theorem \ref{thm:high-dim} is a straightforward corollary of the following stronger Schur-concavity result. For background on Schur-majorisation, we refer for example to \cite{Bh}.
\begin{theorem}\label{thm:high-dim-Schur}
Let $d \geq 5$ and let $\xi_1, \xi_2, \dots$ be independent random vectors uniform on the unit Euclidean sphere $S^{d-1}$ in $\R^d$. For every $n \geq 1$ and $0 < p \leq d-4$, the function
\[
(x_1, \dots, x_n) \mapsto \E\left|\sum_{k=1}^n \sqrt{x_k}\xi_k\right|^{-p}
\]
is Schur-concave on $\R_+^n$.
\end{theorem}
\begin{proof}
Thanks to Lebesgue's monotone convergence theorem, it suffices to show that for every $\delta > 0$, the theorem holds with $|\cdot|^{-p}$ replaced by the function $\Psi_\delta(x) = (|x|^2+\delta)^{-p/2}$. The gain is that $\Psi_\delta$ is $C^\infty$ on $\R^d$. In view of the result of Baernstein II and Culverhouse from \cite{BC}, it suffices to show that $\Psi_\delta$ is bisubharmonic, that is $\Delta\Delta \Psi_\delta \geq 0$ on $\R^d$. We approach this directly. {\red Recall that $\Delta f(|x|)=\frac{d-1}{|x|}f'(|x|)+f''(|x|)$ for a rotation invariant function $f(|x|)$ on $\R^d$, $f \in C^2(\R_+)$}. We have,
\[
\Delta\Delta\Psi_\delta(x) = p(p+2)\left(|x|^2+\delta\right)^{-\frac{p}{2}-4}(A|x|^4+B|x|^2+C),
\]
where $A=(p-d+2)(p-d+4)$, $B=2\delta(d+2)(-p+d-4)$ and $C=\delta^2d(d+2)$. For $p<d-4$, plainly $A>0$ and $B^2-4AC = 8\delta^2(d+2)(p+4)(p-d+4)<0$. This shows that $\Psi_\delta$ is bisubharmonic on $\R^d$ for every $\delta > 0$.
\end{proof}

\begin{remark}\label{rem:bisubharm-equiv}
The crux of Baernstein II and Culverhouse's work is the observation that the bisubharmonicity of a continuous function $\Psi$ on $\R^d$ on one hand is sufficient for the Schur-convexity of the corresponding moment functional from Theorem \ref{thm:high-dim-Schur}, $\E\Psi\left(\sum_{k=1}^n \sqrt{x_k}\xi_k\right)$ (and necessary when $\Psi$ is radial), and on the other hand, it is equivalent to the convexity of the function
\[
t \mapsto \E\Psi(v + \sqrt{t}\xi)
\]
on $\R_+$ for every $v \in \R^d$. In the sequel, we will need to examine the behaviour of this function on $(0,1)$ for unit vectors $v$ when $\Psi(x) = |x|^{-p}$ (see Section \ref{sec:two-coeff} below). 
\end{remark}

\subsection{Outline of the proof of Theorem \ref{thm:d=4}}
Recall that here $d=4$ and $\xi_1, \xi_2, \dots$ are independent random vectors uniform on the unit Euclidean sphere $S^{3}$ in $\R^4$. For notational convenience, we put $q = -p$, $0 < p < 3$ and set
\begin{align}\label{eq:C2p}
C_{2}(p) &= c_{4,2}(q)^q = \E\left|\frac{\xi_1+\xi_2}{\sqrt{2}}\right|^{-p} = 2^{p/2}\frac{\Gamma(3-p)}{\Gamma\left(2-\frac{p}{2}\right)\Gamma\left(3-\frac{p}{2}\right)}, \\\label{eq:Cinfp}
C_{\infty}(p) &= c_{4,\infty}(q)^q = \E\left|\frac{Z}{2}\right|^{-p} = 2^{p/2}\Gamma\left(2-\frac{p}{2}\right),
\end{align}
where $Z$ is a standard Gaussian random vector in $\R^4$ (consult \eqref{eq:const-2} and \eqref{eq:const-inf} to justify the explicit expressions on the right hand sides). Moreover, let $C(p)$ be the best constant such that the following equivalent form of \eqref{eq:khin},
\begin{equation}\label{eq:OPneg}
\E\left|\sum_{k=1}^n a_k\xi_k\right|^{-p} \leq C(p)\left(\sum_{k=1}^n a_k^2\right)^{-p/2}
\end{equation}
holds for every $n \geq 1$ and every real scalars $a_1, \dots, a_n$.

Theorem \ref{thm:d=4} is a consequence of the next two results, where we break it up into two regimes.

\begin{theorem}\label{thm:OPneg-p<2}
For $0 < p \leq 2$, we have $C(p) = C_\infty(p)$.
\end{theorem}

\begin{theorem}\label{thm:OPneg-p>2}
For $2 < p < 3$, we have $C(p) = C_2(p)$.
\end{theorem}

As optimality is clear, for the proofs of these theorems, we need to show that \eqref{eq:OPneg} holds with the specified values of $C(p)$.

\subsubsection{Outline of the proof of Theorem \ref{thm:OPneg-p<2}}\label{sec:beginning-proof-p<2}
Thanks to homogeneity, we can assume that the $a_k$ are positive with $\sum a_k^2 = 1$. Using the Fourier-analytic formula for negative moments \eqref{eq:neg-mom-rot} and H\"older's inequality, we obtain
\begin{align}\label{eq:sum-Holder}
\E\left|\sum_{k=1}^n a_k\xi_k\right|^{-p} &= \kappa_{p,4}\int_0^\infty \left(\prod_{k=1}^n \jj_1(a_kt) \right) t^{p-1} \dd t \\\notag
&\leq \kappa_{p,4}\prod_{k=1}^n \left(\int_0^\infty |\jj_1(a_k t)|^{a_k^{-2}} t^{p-1} \dd t\right)^{a_k^2} \\\notag
&= \kappa_{p,4}\prod_{k=1}^n \left(a_k^{-p}F\left(p, a_k^{-2}\right)\right)^{a_k^2}.
\end{align}
where the following function has emerged (after a change of variables in the last line)
\begin{equation}\label{eq:defF}
F(p,s) = \int_0^\infty |\jj_1(t)|^st^{p-1} \dd t, \qquad p, s  >0.
\end{equation}
This integral is finite as long as $p < \frac{3s}{2}$ because $\jj_1(t) = O(t^{-3/2})$ (see \eqref{eq:j1-w1} below).

The next step is to maximise, individually, the terms in the product on the right hand side of \eqref{eq:sum-Holder}, that is to look into $\sup_{s \geq 1} s^{p/2}F(p,s)$. Heuristically, if we aim at proving that the worst case is Gaussian, that is when $a_1 = \dots = a_n = \frac{1}{\sqrt{n}}$ with $n \to \infty$, a natural candidate for this supremum is then given by $s \to \infty$, which would correspond to the inequality
\begin{equation}\label{eq:perfect}
\begin{split}
s^{p/2}F(p,s) \leq \lim_{s \to \infty} s^{p/2}\int_0^\infty |\jj_1(t)|^st^{p-1} \dd t &= \lim_{s \to \infty} \int_0^\infty |\jj_1(t/\sqrt{s})|^s t^{p-1}\dd t \\
&= \int_0^\infty e^{-t^2/8}t^{p-1} \dd t
\end{split}
\end{equation}
(the last line can be justified using $\jj_1(t) = 1 - \frac{t^2}{8} + o(t^2) = e^{-t^2/8} + o(t^2)${\red , recall the power-series definition \eqref{eq:defjj} of $\jj_1$}). Were it true for all values of $p$ and $s$, we would get
\[
\E\left|\sum_{k=1}^n a_k\xi_k\right|^{-p} \leq \kappa_{p,4} \int_0^\infty e^{-t^2/8}t^{p-1} \dd t = C_\infty(p),
\]
finishing the proof. Unfortunately, the integral inequality \eqref{eq:perfect} fails in certain ranges  of $p$ and $s$, where additional arguments and ideas are needed. This is how we will proceed.

\emph{Step 1: Inequality \eqref{eq:perfect} holds for all $0 < p \leq 2$ and $s \geq 2$.} 

As above, this gives the following partial case of the theorem when all coefficients $a_k$ are small.

\begin{corollary}\label{cor:all-small-p<2}
When $0 < p \leq 2$, inequality \eqref{eq:OPneg} holds with $C(p) = C_\infty(p)$ for every $n \geq 1$ and all real numbers $a_1, \dots, a_n$ with $\max_{k \leq n} |a_k| \leq \frac{1}{\sqrt{2}}\left(\sum_{k=1}^n a_k^2\right)^{1/2}$.
\end{corollary}

\emph{Step 2: For $\frac14 \leq p \leq 2$, we employ induction on $n$ to cover the case $\max_{k \leq n} |a_k| > \frac{1}{\sqrt{2}}\left(\sum_{k=1}^n a_k^2\right)^{1/2}$.} 

This will give the theorem when $p \geq \frac14$. For the induction to work, \eqref{eq:OPneg} is strengthened, but the base of the induction fails for small $p$ (roughly $p < 0.2$), hence the next two steps. Fortunately, when $p$ is \emph{small}, the integral inequality holds for a wider range of $s$.

\emph{Step 3:  Inequality \eqref{eq:perfect} holds for all $0 < p \leq \frac14$ and $s \geq 1.3$.} 

\begin{corollary}\label{cor:all-small-p-small}
When $0 < p \leq \frac14$, inequality \eqref{eq:OPneg} holds with $C(p) = C_\infty(p)$ for every $n \geq 1$ and all real numbers $a_1, \dots, a_n$ such that $\max_{k \leq n} |a_k| \leq \sqrt{\frac{10}{13}}\left(\sum_{k=1}^n a_k^2\right)^{1/2}$.
\end{corollary}

Finally, when one of the coefficients $a_k$ is \emph{large}, the inequality holds for a different reason (we will use a sort of projection-type argument).

\emph{Step 4: When $0 < p \leq \frac14$,  inequality \eqref{eq:OPneg} holds with $C(p) = C_\infty(p)$ for every $n \geq 1$ and all real numbers $a_1, \dots, a_n$ with $\max_{k \leq n} |a_k| > \sqrt{\frac{10}{13}}\left(\sum_{k=1}^n a_k^2\right)^{1/2}$. }

\subsubsection{Outline of the proof of Theorem \ref{thm:OPneg-p>2}}\label{sec:beginning-proof-p>2}
If we want to prove that the worst case is now $n=2$ with $a_1 = a_2 = \frac{1}{\sqrt{2}}$, it is only natural to expect that $\sup_{s \geq 1} s^{p/2}F(p,s)$ is attained at $s=2$, corresponding to the integral inequality
\begin{equation}\label{eq:perfect1}
s^{p/2}F(p,s) \leq 2^{p/2}F(p,2).
\end{equation}
We will proceed similarly, with only the first two steps sufficing, as the inductive base now holds in the entire range.

\emph{Step 1: Inequality \eqref{eq:perfect1} holds for all $2 < p < 3$ and $s \geq 2$.} 

Taking this statement for granted for now, we derive the following corollary.

\begin{corollary}\label{cor:all-small}
When $2 < p < 3$, inequality \eqref{eq:OPneg} holds with $C(p) = C_2(p)$ for every $n \geq 1$ and all real numbers $a_1, \dots, a_n$ with $\max_{k \leq n} |a_k| \leq \frac{1}{\sqrt{2}}\left(\sum_{k=1}^n a_k^2\right)^{1/2}$.
\end{corollary}
\begin{proof}
Assuming $\sum a_k^2 = 1$ and applying \eqref{eq:perfect1} to the right hand side of \eqref{eq:sum-Holder} yields
\begin{align*}
\E\left|\sum_{k=1}^n a_k\xi_k\right|^{-p} \leq \kappa_{p,4}\cdot 2^{p/2}F(p,2) &= 2^{p/2}\kappa_{p,4}\int_0^\infty \jj_1(t)^2t^{p-1}\dd t\\
&= 2^{p/2}\E\left|\xi_1+\xi_2\right|^{-p} = C_2(p)
\end{align*}
(for the penultimate step, recall {\red again \eqref{eq:sum-Holder}}).
\end{proof}

\emph{Step 2: For $2 < p < 3$, we employ induction on $n$ to cover the case $\max_{k \leq n} |a_k| > \frac{1}{\sqrt{2}}\left(\sum_{k=1}^n a_k^2\right)^{1/2}$.} 

To carry out these steps, we first establish a variety of indispensable technical estimates. After this has been done in the next section, we will conclude the proof in Sections \ref{sec:end-proof-p<2} and \ref{sec:end-proof-p>2}.

\section{Ancillary results}

\subsection{Two-coefficient function}\label{sec:two-coeff}

By rotational invariance, 
\[
\E|a_1\xi_1 + a_2\sqrt{t}\xi_2|^{-p} = \E|a_1e_1+a_2{\red \sqrt{t}}\xi_2|^{-p}.
\] 
We begin with some properties of the function $t \mapsto \E|a_1e_1+a_2{\red \sqrt{t}}\xi_2|^{-p}$, particularly important in the inductive part of our proof. Recall the definition of the (Gaussian) hypergeometric function which shows up very naturally, as explained in the next lemma. For real parameters $a,b,c$, it is defined for $|z|<1$ by the power series,
\[
{}_2F_1(a, b; c; z) = \sum_{k=0}^\infty\frac{(a)_k(b)_k}{(c)_k}\frac{z^k}{k!}.
\]

\begin{lemma}\label{lm:two-coeff-expansion}
Let $d \geq 1$ and let $\xi$ be a random vector uniform on the unit Euclidean sphere $S^{d-1}$ in $\R^d$. Let $p < d-1$. Then
\begin{align*}
\E|e_1 + \sqrt{t}\xi|^{-p} &= {}_2F_1\left(\frac{p}{2}, \frac{p-d+2}{2}; \frac{d}{2}; t\right) \\
&= \sum_{k=0}^\infty \frac{\left(\frac{p}{2}\right)_{k}\left(\frac{p-d+2}{2}\right)_{k}}{\left(\frac{d}{2}\right)_{k}}\frac{t^k}{k!}, \qquad 0 < t < 1.
\end{align*}
\end{lemma}

\begin{proof}
Fix $0 < t < 1$. Let $\theta = \scal{e_1}{\xi}$ be the first coordinate of $\xi$. Thus
\begin{align*}
\E|e_1 + \sqrt{t}\xi|^{-p} &= \E(1+2\sqrt{t}\theta + t)^{-p/2} \\
&= (1+t)^{-p/2}\E\left(1 + \frac{2\sqrt{t}}{1+t}\theta\right)^{-p/2} \\
&= (1+t)^{-p/2}\sum_{k=0}^\infty \binom{-p/2}{2k}(\E \theta^{2k})\left(\frac{2\sqrt{t}}{1+t}\right)^{2k}.
\end{align*}
From \eqref{eq:jj},
\[
\E \theta^{2k} = \frac{(2k)!}{2^{2k}\cdot k!(d/2)_k},
\]
hence
\[
\E|e_1 + \sqrt{t}\xi|^{-p} =(1+t)^{-p/2}\sum_{k=0}^\infty  \frac{(p/2)_{2k}}{2^{2k}(d/2)_k}\frac{1}{k!}\left(\frac{4t}{(1+t)^2}\right)^{k}.
\]
Since $\left(p/2\right)_{2k}2^{-2k} = \left(\frac{p}{4}\right)_k\left(\frac{p+2}{4}\right)_k$, we get
\begin{align*}
\E|e_1 + \sqrt{t}\xi|^{-p} &= (1+t)^{-p/2}{}_{2}F_1\left(\frac{p}{4},\frac{p+2}{4};\frac{d}{2};\frac{4t}{(1+t)^2}\right) \\
&= {}_{2}F_1\left(\frac{p}{2},\frac{p-d+2}{2};\frac{d}{2};t\right),
\end{align*}
where the last identity follows from Kummer's quadratic transformations for the hypergeometric function ${}_{2}F_1$ (see, e.g. 15.3.26 in \cite{AS}). The desired power series expansion now follows from the definition of ${}_2F_1$.
\end{proof}

This in particular yields the explicit expression for $c_{d,2}(q)$ from \eqref{eq:const-2}.

\begin{corollary}\label{cor:const-2}
For $d \geq 1$ and $p < d-1$, we have
\[
\E|\xi_1+\xi_2|^{-p} = {}_2F_1\left(\frac{p}{2}, \frac{p-d+2}{2};\frac{d}{2};1\right) = \frac{\Gamma\left(\frac{d}{2}\right)\Gamma(d-p-1)}{\Gamma\left(\frac{d-p}{2}\right)\Gamma\left(d-\frac{p}{2}-1\right)}.
\]
\end{corollary}
\begin{proof}
The expression on the right hand side follows from Gauss' summation identity (see, e.g. 15.1.20 in \cite{AS}).
\end{proof}

\begin{remark}\label{rem:alt-proof}
In addition to the proof of Lemma \ref{lm:two-coeff-expansion} presented above we would like to sketch a different argument, in the spirit of Lemma 1 from \cite{BC}, which bypasses the explicit use of the hypergeometric function. Let $\Psi(x) = |x|^{-p}$. Since on the unit sphere $\xi \in S^{d-1}$ is the outer-normal, by the divergence theorem (for the usual Lebesgue \emph{non}normalised surface integral),
\begin{align*}
\frac{\dd}{\dd t}\int_{S^{d-1}}|e_1 + \sqrt{t}\xi|^{-p} \dd \xi &= \frac{1}{2\sqrt{t}}\int_{S^{d-1}} \scal{(\nabla \Psi)(e_1 + \sqrt{t}\xi)}{\xi} \dd \xi \\
&=\frac{1}{2\sqrt{t}}\int_{\mathbb{B}^d} \text{div}_x\Big((\nabla \Psi)(e_1 + \sqrt{t}x)\Big) \dd x \\
&= \frac{1}{2}\int_{\mathbb{B}^d} (\Delta\Psi)(e_1 + \sqrt{t}x) \dd x
\end{align*}
for every $0 < t < 1$ (note that $e_1 + \sqrt{t}x$ on $B_2^d$ is away from the origin where $\Psi$ is singular). Computing the Laplacian yields the identity
\[
\frac{\dd}{\dd t}\int_{S^{d-1}}|e_1 + \sqrt{t}\xi|^{-p} \dd \xi = \frac{p(p-d+2)}{2}\int_{B_2^d} |e_1 + \sqrt{t}x|^{-p-2} \dd x.
\]
Writing the last integral using polar coordinates allows to compute the higher derivatives by simply iterating this identity.
Thus
\begin{align}\label{eq:two-coeff-der}
\frac{\dd}{\dd t} \E|e_1 + \sqrt{t}\xi|^{-p} &= \frac{p(p-d+2)}{2}\frac{1}{|S^{d-1}|}\int_{B_2^d}|e_1+\sqrt{t}x|^{-p-2}\dd x \\\notag
&= \frac{p(p-d+2)}{2}\frac{1}{|S^{d-1}|}\int_0^1\int_{S^{d-1}}r^{d-1}|e_1+\sqrt{tr^2}\xi|^{-p-2}\dd \xi
\end{align}
and
\begin{align*}
\frac{\dd^2}{\dd t^2} \E|e_1 + \sqrt{t}\xi|^{-p} = &\frac{p(p-d+2)}{2}\frac{(p+2)(p-d+4)}{2}\\
&\cdot\frac{1}{|S^{d-1}|}\int_0^1r^{d+1}\int_{\mathbb{B}^d}|e_1+\sqrt{t}rx|^{-p-4}\dd x\dd r,
\end{align*}
etc. It then remains to evaluate these derivatives at $t = 0$ to get the  power-series expansion coefficients.
\end{remark}

\begin{corollary}\label{cor:two-coeff-R4}
Let $\xi$ be a random vector uniform on the unit Euclidean sphere $S^{3}$ in $\R^4$. Let $0 < p \leq 2$. Then
\[
\E|e_1 + \sqrt{t}\xi|^{-p} \leq 1 - \frac{p(2-p)}{8}t - \frac{p^2(4-p^2)}{192}t^2, \qquad 0 < t < 1.
\]
\end{corollary}
\begin{proof}
When $d=4$ and $0 < p < 2$, all the terms in the power series from Lemma~\ref{lm:two-coeff-expansion} but the first one (which equals $1$) are negative. Dropping all but the first three thus gives the desired bound.
\end{proof}

\begin{corollary}\label{cor:two-coeff-min}
Let $d \geq 1$. Let $\xi$ be a random vector uniform on the unit Euclidean sphere $S^{d-1}$ in $\R^d$. Let $0 < p \leq d-2$. Then for every vector $v$ in $\R^d$ and $a > 0$, we have
\[
\E|v + a\xi|^{-p} \leq \min\{|v|^{-p},a^{-p}\}.
\]
\end{corollary}
\begin{proof}
By homogeneity and rotational invariance, we can assume without loss of generality that $v = e_1$ and $0 < a < 1$ {\red (note in particular that rotational invariance implies $\E|e_1+a\xi|^{-p} = \E|\xi_1+a\xi_2|^{-p} = \E|ae_1+\xi|^{-p}$, so the case $a > 1$ reduces to the case $0 < a < 1$ by multiplying both sides by $a^p$)}. From \eqref{eq:two-coeff-der} we see that the function $a \mapsto \E|e_1 + a\xi|^{-p}$ is nonincreasing, in particular $\E|e_1 + a\xi|^{-p} \leq 1$.
\end{proof}

\subsection{Bounds for the inductive base}

We remark that in several places we need to use numerical values of some special functions such as $\jj_1$, $\Gamma$, $\psi = (\log \Gamma)'$ and will implicitly do so (to the required precision).

Based on tables left by Gauss, Deming and Colcord in \cite{DC} found the value of $\min_{x>0} \Gamma(x)$ correct up to the 19th decimal which we record here (although we will not require such precision). 

\begin{lemma}[\cite{DC}]\label{lm:min-gamma}
We have, 
\[
\min_{x>0} \Gamma(x) = 0.8856031944108886887..,
\] 
uniquely occurring at $x_0 = 1.46163214496836226..$.
\end{lemma}

To check the base of the induction from Step 2 in Section \ref{sec:beginning-proof-p<2}, we will need the following two-point inequality.

\begin{lemma}\label{lm:ind-base}
For every $\frac18 \leq q \leq 1$ and $0 \leq t \leq 1$, we have
\[
1 - \frac{q(1-q)}{2}t - \frac{q^2(1-q^2)}{12}t^2 \leq \Gamma(2-q)\left(2 - \left(\frac{3-t}{2}\right)^{-q}\right).
\]
\end{lemma}
\begin{proof}
We let $Q_q(t), R_q(t)$ denote the left hand side and the right hand side respectively and set $h_q(t) = R_q(t) - Q_q(t)$. We examine its second derivative,
\[
h_q''(t) = -2^q\Gamma(2-q)q(q+1)(3-t)^{-q-2} + \frac{q^2(1-q^2)}{6}
\]
which is clearly decreasing in $t$. Therefore, for all $0 \leq t \leq 1$, $h_q''(t) \leq h_q''(0)$ and for $0 < q < 1$, with the aid of Lemma \ref{lm:min-gamma},
\begin{align*}
-\frac{3^{q+2}}{2^q\cdot q(1+q)} h_q''(0) &= \Gamma(2-q) - (3/2)^{q+1}q(1-q) \\
&> 0.88 - (3/2)^2\cdot \frac{1}{4} = 0.3175.
\end{align*}
As a result, $h_q(t)$ is concave on $[0,1]$. To show that $h_q(t) \geq 0$ on $[0,1]$, it thus suffices to verify that (A) $h_q(0) \geq 0$ and (B) $h_q(1) \geq 0$, for all $\frac18 \leq q \leq 1$.

(A): $h_q(0) \geq 0$ is equivalent to $\Gamma(2-q)\left(2-(2/3)^q\right) \geq 1$, or after taking logarithms, $g(q) \geq f(q)$ with $g(q) = \log\Gamma(2-q)$, $f(q) = -\log 2 - \log(1-\frac12(\frac23)^q)$. Both $f$ and $g$ are clearly convex (note $f(q) = -\log 2 + \sum_{k=1}^\infty [\frac12(\frac23)^q]^k/k$). For $\frac18 \leq q \leq 0.35$, we lower-bound $g$ by its supporting tangent at $q = \frac18$, $g(q) \geq \ell(q) = g(\frac18) + g'(\frac18)(q-\frac18)$. Since $\ell(\frac18) - f(\frac18) > 0.0005$ and $\ell(0.35) - f(0.35) > 0.0003$, thanks to the convexity of $f$, we conclude that indeed $g(q) > f(q)$ for $\frac18 \leq q \leq 0.35$. For the remaining range $0.35 \leq q \leq 1$, we crudely have, using the monotonicity of $f$ and Lemma \ref{lm:min-gamma},
\[
f(q) \leq f(0.35) < -0.124 < \log(0.885) <  \log \Gamma(2-q) = g(q).
\]

(B): $h_q(1) \geq 0$ is equivalent to $\Gamma(2-q) \geq 1 - \frac{q(1-q)}{2} - \frac{q^2(1-q^2)}{12}$. Taking the logarithms and using $\log(1-x) \leq -x$, $x < 1$, it suffices to show that
\[
f(q) = \log\Gamma(2-q) +  \frac{q(1-q)}{2} + \frac{q^2(1-q^2)}{12}
\]
is nonnegative. This in fact holds for all $0 \leq q \leq 1$. Indeed, $f(0) = f(1) = 0$ and for $0 \leq q \leq 1$,
\[
f''(q) = \sum_{k=0}^\infty \frac{1}{(2-q+k)^2} -q^2 - \frac{5}{6}.
\]
{\red It suffices to show that this is negative for $0 \leq q \leq 1$ so that the concavity of $f$ will finish the argument. To this end, we upper bound the convex function $h(q) = \sum_{k=0}^\infty \frac{1}{(2-q+k)^2}$ by linear chords. For $0 \leq q \leq \frac12$, we have, $h(q) \leq h_1(q) = \frac{\frac12-q}{\frac12}h(0) + \frac{q}{\frac12}h(\frac12)$ and since $h(0) = \frac{\pi^2}{6}-1$, $h(\frac12) = \frac{\pi^2}{2}-4$, we get $h_1(q) = \frac{2}{3}(\pi^2-9)q+\frac{\pi^2}{6}-1$. We check that $h_1(q)-q^2-\frac56$ is maximised at $q = \frac{\pi^2-9}{3}$ with the value less than $-0.1$. For $\frac{1}{2} \leq q \leq 1$, we have $h(q) \leq h_2(q) = \frac{1-q}{\frac12}h(\frac12) + \frac{q-\frac12}{\frac12}h(1)$ and since $h(1) = \frac{\pi^2}{6}$, we get $h_2(q) = 2(\frac{12-\pi^2}{3})q+\frac{5\pi^2}{6}-8$. Finally, we check that $h_2(q)-q^2-\frac{5}{6}$ is maximised at $q = \frac{12-\pi^2}{3}$ with the value also less than $-0.1$.}
\end{proof}

We emphasise that in part (B) of this proof, we have shown that when $t=1$, the inequality in Lemma \ref{lm:ind-base} holds for all $0 \leq q \leq 1$. This combined with Corollary \ref{cor:two-coeff-R4} leads to the following result, important in the sequel in the proof of integral inequality \eqref{eq:perfect}.

\begin{corollary}\label{cor:ind-base-t=1}
Let $\xi$ be a random vector uniform on the unit Euclidean sphere $S^{3}$ in $\R^4$. Let $0 < p \leq 2$. Then
\[
\E|e_1 + \xi|^{-p} \leq \Gamma\left(2 - \frac{p}{2}\right),
\]
equivalently
\begin{equation}\label{eq:Hs=2}
\int_0^\infty |\jj_1(t)|^2 t^{p-1}\dd t \leq 2^{p-1}\Gamma(p/2).
\end{equation}
\end{corollary}
\begin{proof}
To explain the equivalent form involving $\jj_1$, note that, $\E|e_1 + \xi|^{-p} = \E|\xi + \xi'|^{-p}$, for an independent copy $\xi'$ of $\xi$, thanks to rotational invariance. It remains to use \eqref{eq:neg-mom-rot} which gives $\E|\xi_1+\xi_2|^{-p} = \kappa_{p,4}\int_0^\infty |\jj_1(t)|^2 t^{p-1}\dd t$ and plug in the value of $\kappa_{p,4}$.
\end{proof}

\subsection{The integral inequality: $0 < p \leq 2$}

We record for future use the following bounds
\begin{align}\label{eq:j1-exp}
|\jj_1(t)| &\leq \exp\left(-\frac{t^2}{8}-\frac{t^4}{3\cdot 2^7}\right), \qquad 0 \leq t \leq 4, \\\label{eq:j1-w1}
|\jj_1(t)| &\leq (8/\pi)^{1/2}t^{-1}(t^2-1)^{-1/4}, \qquad t \geq 1,
\end{align}
where the first one appears as Lemma 3.1 in \cite{OP} (see also \cite[Lemma 3.6]{Brz} for the proof of a more general statement) and the second one can be found in Watson's treatise (see \cite[p.447]{W} as well as \cite[Lemma 4.4]{Dirk-cyl}), which in particular gives
\begin{equation}\label{eq:j1-w2}
|\jj_1(t)| \leq (8/\pi)^{1/2}\left(\frac{t_0^2}{t_0^2-1}\right)^{1/4}t^{-3/2}, \qquad t \geq t_0 \geq 1.
\end{equation}

We define
\begin{equation}\label{eq:defH}
H(p, s) = \int_0^\infty \Big(e^{-st^2/8}-|\jj_1(t)|^s\Big)t^{p-1} \dd t, \qquad 0 < p < 2,\  s > 1
\end{equation}
and immediately observe that after a change of variables one integral can be expressed in terms of the gamma function,
\begin{equation}\label{eq:defG}
G(p,s) = \int_0^\infty e^{-st^2/8}t^{p-1} \dd t = s^{-p/2}2^{3p/2-1}\Gamma(p/2).
\end{equation}
Recall \eqref{eq:defF}, $F(p,s) = \int_0^\infty |\jj_1(t)|^st^{p-1}\dd t$, so
\begin{equation}\label{eq:HasFG}
H(p,s) = G(p,s) - F(p,s).
\end{equation}
Then the crucial integral inequality \eqref{eq:perfect} is equivalent to $H(p,s) \geq 0$.

Our main goal and result here is that the integral inequality $H(p,s) > 0$ holds in rather wide ranges of parameters $(p,s)$ (however, it does not {\red hold for all} $0 < p < 2$ and $s > 1$ which, as already noted, would have been enough to deduce Theorem \ref{thm:OPneg-p<2}).

\begin{lemma}\label{lm:H}
The inequality $H(p,s) > 0$ holds in the following cases

(a) $0 < p \leq 2$ and $s \geq 2$,

(b) $0 < p \leq \frac{1}{4}$ and $s \geq 1.3$.
\end{lemma}

For the proof, we will need several rather intricate estimates on various integrals. The general idea we employ here follows \cite{OP} and is to first use the explicit bounds on $\jj_1$ from \eqref{eq:j1-exp} and \eqref{eq:j1-w2} to get $H > 0$ in certain but \emph{not all} cases and then extend them by interpolating in $s$ (exploiting the simple dependence of $G$ on $s$). This is in contrast to several works, e.g. \cite{Brz, CGT, CKT, Dirk-cyl, Ko, Mo} which heavily rely on the approach developed by Nazarov and Podkorytov in \cite{NP} to integral inequalities with oscillatory integrands. We also refer to recent papers  \cite{ALM} as well as \cite{MelRob} for connections between such integral inequalities and majorisation.

We begin by setting
\begin{equation}\label{eq:defU}
\begin{split}
U(p,s) &= \frac{4^p(2\pi\cdot15^{1/2})^{-s/2}}{3s/2-p} \\
&\qquad+ 2^{3p/2-1}s^{-p/2}\left(\Gamma\left(p/2\right) - \frac{\Gamma(p/2+2)}{6s} + \frac{\Gamma(p/2+4)}{72s^2}\right)
\end{split}
\end{equation}
which emerges in the next lemma (following Lemma 3.2 from \cite{OP}).
\begin{lemma}\label{lm:j-U}
For $p < 3s/2$, we have
\[
F(p,s) < U(p,s).
\]
\end{lemma}
\begin{proof}
Using \eqref{eq:j1-exp} and \eqref{eq:j1-w2} with $t_0 = 4$, 
we get
\begin{align*}
F(p,s) = \int_0^\infty |\jj_1(t)|^st^{p-1} \dd t < &\int_0^\infty \exp\left(-s\frac{t^2}{8}-s\frac{t^4}{3\cdot 2^7}\right)t^{p-1} \dd t \\
&\qquad\qquad\qquad+ \left(\frac{8}{15^{1/4}(2\pi)^{1/2}}\right)^s\frac{4^{p-3s/2}}{3s/2-p},
\end{align*}
valid for $p < \frac{3s}{2}$. After the change of variables $u=st^2/8$, the first integral becomes
\[
2^{3p/2-1}s^{-p/2}\int_0^\infty e^{-\frac{u^2}{6s}}e^{-u}u^{p/2-1}\dd u.
\]
We estimate the first exponential using $e^{-x} \leq 1-x+\frac{x^2}{2}$, $x \geq 0$, which gives the bound
\[
\int_0^\infty \left(1 - \frac{u^2}{6s}+\frac{u^4}{72s^2}\right)e^{-u}u^{p/2-1}\dd u  = \Gamma\left(p/2\right) - \frac{\Gamma(p/2+2)}{6s} + \frac{\Gamma(p/2+4)}{72s^2}
\]
{\red on the integral appearing in the above expression}.
\end{proof}

\begin{lemma}\label{lm:U-G}
The inequality
\[
U(p,s) < G(p,s)
\]
holds in the following cases

(i) $0 < p \leq \frac{1}{4}$ and $s \geq \frac{17}{10}$,

(ii) $0 < p \leq \frac45$ and $s \geq 2$,

(iii) $0 \leq p \leq 2$ and $s \geq \frac{8}{3}$.
\end{lemma}
\begin{proof}
Note that $U < G$ is equivalent to the following inequality (after cancelling {\red the terms containing} $\Gamma(p/2)$ on both sides, factoring out $\Gamma(p/2+2)$ and moving terms across using that $3s/2-p > 0$),
\begin{equation}\label{eq:U1}
(2\pi\cdot 15^{1/2})^{s/2}2^{-p/2}\left(\frac{3s}{2}-p\right)\frac{12s-\left(\frac{p}{2}+2\right)\left(\frac{p}{2}+3\right)}{144} > \frac{s^{p/2+2}}{\Gamma(p/2+2)}.
\end{equation}
To shorten the notation, let $a = (2\pi)^{1/2}\cdot 15^{1/4}$ and
\[
A(p,s) = 2^{-p/2}\left(\frac{3s}{2}-p\right)\frac{12s-\left(\frac{p}{2}+2\right)\left(\frac{p}{2}+3\right)}{144}
\]
which is decreasing in $p$ and increasing in $s$. In each of the cases we will simply replace $A$ with its smallest possible value given the range of $p$ and $s$, so we let $p_1 = \frac14$, $s_1 = \frac{17}{10}$, $p_2 = \frac45$, $s_2 = 2$ and $p_3 = 2$, $s_3 = \frac83$ and have $A(p,s) \geq A_k$, where $A_k = A(p_k, s_k)$ for $k=1,2,3$ in cases (i), (ii), (iii), respectively. Then it suffices to prove that
\[
A_ka^s > \frac{s^{p/2+2}}{\Gamma(p/2{\red +2})}.
\]
We take the logarithm and consider 
\[
f(p,s) = s\log a+\log A_k-\left(\frac{p}{2}+2\right)\log s + \log \Gamma\left(\frac{p}{2}+2\right).
\]
Our goal is to show that $f(p,s) > 0$. We observe that
\[
\frac{\partial}{\partial p}f(p,s) = -\frac{1}{2}\log s + {\red \frac{1}{2}}\psi\left(\frac{p}{2}+2\right) \leq -\frac{1}{2}\log s_k + {\red \frac{1}{2}}\psi\left(\frac{p_k}{2}+2\right)
\]
in each case respectively and the resulting numerical values on the right bounded above by {\red $-0.015$, $-0.02$ and $-0.029$}, $k = 1, 2, 3$. Similarly,
\[
\frac{\partial}{\partial s}f(p,s) = \log a - \frac{p/2+2}{s} \geq \log a - \frac{p_k/2+2}{s_k}
\]
with the right hand side bounded this time below by $0.34$, $0.39$ and $0.47$, $k = 1, 2, 3$. Thus $f(p,s)$ is decreasing in $p$ and increasing in $s$, so 
\[
f(p,s) \geq f(p_k, s_k)
\]
and after plugging in the explicit numerical values, the right hand side is bounded below by $0.041$, $0.049$ and $0.032$, $k = 1, 2, 3$, thus proving (i), (ii) and (iii).
\end{proof}

The next two lemmas are vital for the interpolation argument.

\begin{lemma}\label{lm:interp>2}
For $\frac45 \leq p \leq 2$, we have
\[
F(p, 8/3) < e^{-p/6}G(p,2).
\]
\end{lemma}
\begin{proof}
Using
\eqref{eq:j1-w2} with $t_0 = 5$,
we get
\begin{equation}\label{eq:83-1}
\int_5^\infty |\jj_1(t)|^{8/3} t^{p-1} \dd t \leq (8/\pi)^{4/3}(25/24)^{2/3}\frac{5^{p-4}}{4-p} 
\end{equation}
which for $p \leq 2$ gives
\begin{align*}
\int_5^\infty |\jj_1(t)|^{8/3} t^{p-1} \dd t \leq (8/\pi)^{4/3}(25/24)^{2/3}\frac{5^{p-4}}{2} = \frac{2}{3^{2/3}\cdot 5^{8/3} \pi^{4/3}}5^p.
\end{align*}
We divide the interval $[0,5]$ into consecutive subintervals of the form $[\frac{k}{m},\frac{k+1}{m}]$, for $k = 0, 1, \dots, 5m-1$ with $m = 100$ and crudely bound
\begin{equation}\label{eq:83-2}
\begin{split}
\int_0^5 &|\jj_1(t)|^{8/3} t^{p-1} \dd t \\
&< \int_{0}^{1/m}t^{p-1}\dd t +  \frac{1}{m}\sum_{k=1}^{5m-1}\max\left\{\left|\jj_1\left(\frac{k}{m}\right)\right|^{8/3}, \left|\jj_1\left(\frac{k+1}{m}\right)\right|^{8/3}\right\}\\
&\qquad\qquad\qquad\qquad\qquad\qquad\cdot\max\left\{\left(\frac{k}{m}\right)^{p-1}, \left(\frac{k+1}{m}\right)^{p-1}\right\}
\end{split}
\end{equation}
(we have used that $|\jj_1|<1$ and that {\red $\jj_1$} is monotone on $[0,5]$, the {\red former justified by \eqref{eq:jj}} and the latter e.g. in \cite{OP}, p. 290, in the proof of Proposition 1.1).
Now, $\int_{0}^{1/m}t^{p-1}\dd t = \frac{1}{p m^p} < \frac{1}{0.8m^p}$. A resulting bound on $e^{p/6}2^{1-p}\int_0^\infty |\jj_1(t)|^{8/3} t^{p-1} \dd t $ is of the form
\[
h(p) = \sum_k \lambda_ka_k^p
\]
with explicit positive numbers $\lambda_k$, $a_k$. We check that $L(p) = \log h(p) < \log \Gamma(p/2) = R(p)$ for $0.8 \leq p \leq 2$ relying on the fact that both sides are clearly convex (recall that summation preserves log-convexity). Specifically, we divide the interval $[0.8, 2]$ into $12$ consecutive subintervals $[u_i, u_{i+1}]$, $u_i = 0.8+0.1i$, $i = 0, 1, \dots, 11$ and on each interval we lower-bound $R(p)$ by its tangent put at the middle $v_i = \frac{u_i+u_{i+1}}{2}$, $\ell_i(p) = R'(v_i)(p-v_i)+R(v_i)$ and then check that $\ell_i(p) > L(p)$ by checking the values at the end-points $p = u_i, u_{i+1}$, which are gathered in Table~\ref{tab:ell-L}. 
\end{proof}

\begin{table}[!hb]
\begin{center}
\caption{Proof of Lemma \ref{lm:interp>2}: lower bounds on the differences at the end-points of the linear approximations $\ell_i$ to $R(p)$.}
\label{tab:ell-L}
\begin{tabular}{r|cccccccccccc}
$i$ & 0 & 1 & 2 & 3 & 4 & 5 & 6 & 7 & 8 & 9 & 10 & 11 \\\hline
$10^3\cdot(\ell_i(u_i)-L(u_i))$ & $1$ & $5$ & $8$ & $9$ & $10$ & $12$ & $13$ & $14$ & $14$ & $15$ & $15$ & $15$  \\ 
$10^3\cdot(\ell_i(u_{i+1})-L(u_{i+1}))$ & $4$ & $8$ & $9$ & $10$ & $11$ & $13$& $14$ & $14$ & $15$ & $15$ & $15$ & $14$
\end{tabular}
\end{center}
\end{table}

\begin{lemma}\label{lm:interp<2}
For $0 < p \leq \frac{1}{4}$, we have
\[
F(p, 1.3) < e^{2p/17}G(p, 1.7).
\]
\end{lemma}

\begin{proof}
Fix $0 < p \leq \frac14$. We break the integral on the left hand side into the sum of $4$ integrals $A_1+\dots+A_4$ over $(0,1)$, $(1,5)$, $(5,10)$ and $(10,\infty)$. For the first one, we use {\red \eqref{eq:j1-exp}},
\[
|\jj_1(t)|^{1.3} < \exp\left\{-\frac{13}{10}\left(\frac{t^2}{8}+\frac{t^4}{3\cdot 2^7}\right)\right\} < 1 - \frac{13}{80}t^2 + \frac{377}{38400}t^4, \qquad 0 < t < 1
\]
(the last inequality obtained by taking the first terms in the power series expansion of the penultimate expression, which gives an upper bound as can be checked directly by differentiation). Integrating against $t^{p-1}$ yields
\[
A_1 \leq \frac{1}{p} - \frac{13}{80(p+2)} + \frac{377}{38400(p+4)} < \frac{1}{p} - \frac{13}{80(p+2)} + \frac{377}{38400\cdot 4}.
\]
For the last one, we use \eqref{eq:j1-w2} with $t_0 = 10$,
\begin{align*}
A_4 \leq \int_{10}^\infty \Big((8/\pi)^{1/2}(100/99)^{1/4}t^{-3/2}\Big)^{1.3}t^{p-1}\dd t &= \frac{2^{53/20}}{11^{13/40}\cdot 5^{3/10}(3\pi)^{13/20}}\frac{10^p}{39-20p} \\
&\leq \frac{2^{53/20}}{11^{13/40}\cdot 5^{3/10}(3\pi)^{13/20}}\frac{10^p}{34}.
\end{align*}
For $A_2$ and $A_3$, we use Riemann sums. First, without any error term thanks to the monotonicity of $j_1$ on $(1,5)$,
\begin{align*}
A_2 &\leq \sum_{k=0}^{4m-1} \max\left\{\left|\jj_1\left(1+\frac{k}{m}\right)\right|^{1.3},\left|\jj_1\left(1+\frac{k+1}{m}\right)\right|^{1.3}\right\}\int_{1+\frac{k}{m}}^{1+\frac{k+1}{m}} t^{p-1}\dd t \\
&<\sum_{k=0}^{4m-1} \max\left\{\left|\jj_1\left(1+\frac{k}{m}\right)\right|^{1.3},\left|\jj_1\left(1+\frac{k+1}{m}\right)\right|^{1.3}\right\}\frac{(1+k/m)^{p-1}}{m}.
\end{align*}
Second, on $(5,10)$, we choose the midpoints and bound the error simply using the supremum of the derivative via the crude (numerical) bound
\[
\sup_{t \in [5,10]} \left|\frac{\dd}{\dd t}|\jj_1(t)|^{1.3}\right| < 0.06
\]
(since $\left|\frac{\dd}{\dd t}|\jj_1(t)|^{1.3}\right| = 1.3|\jj_1(t)|^{0.3}|\jj_1'(t)|$ and $j_1'(t) = -2\frac{J_2(t)}{t} = 2\frac{J_0(t)}{t}-4\frac{J_1(t)}{t^2}$, the function under the supremum can be expressed in terms of $J_0$ and $J_1$ and  the supremum can be estimated by employing the precise polynomial-type approximations to $J_0$ and $J_1$ from \cite{AS}, 9.4.3 and 9.4.6, pp.369--370).
This leads to
\begin{align*}
A_3 &\leq \sum_{k=0}^{5m-1} \left|\jj_1\left(5+\frac{k+1/2}{m}\right)\right|^{1.3}\int_{5+\frac{k}{m}}^{5+\frac{k+1}{m}} t^{p-1}\dd t + 0.06\frac{1}{2m}\int_5^{10}t^{p-1}\dd t \\
&< \sum_{k=0}^{5m-1} \left|\jj_1\left(5+\frac{k+1/2}{m}\right)\right|^{1.3}\frac{(5+k/m)^{p-1}}{m} + \frac{3\cdot 5^p}{100m}.
\end{align*}
With hindsight, we choose $m = 200$. Adding these $4$ estimates together (call the right-most hand sides of these bounds $B_1, \dots, B_4$) and multiplying through $p$, it suffices to show that $L(p) < R(p)$ for $0 < p \leq \frac{1}{4}$, where
\[
L(p) = p\cdot(B_1+\dots+B_4), \qquad R(p) = (e^{2/17}2^{3/2}1.7^{-1/2})^p\Gamma\left(\frac{p}{2}+1\right).
\]
Plainly, $R(p)$ is convex (as being log-convex), whilst
\[
L(p) = \frac{67}{80}+\frac{13}{40(p+2)} + \frac{377}{153600}p + c_1\cdot p10^p + c_2\cdot p5^p + \sum_i \lambda_i pa_i^p
\]
with positive constant $c_1, c_2, \lambda_i$ (specified above) and $a_i \geq 1$ (of the form $(1+k/m)$, $k \geq 0$). Thus, $L(p)$ is also convex and now we proceed similarly to what we did in the proof of Lemma \ref{lm:interp>2}. Note that $L(0) = R(0) = 1$. For $0< p \leq 0.02$, we lower-bound, $R(p) \geq \ell_0(p) = 1 + R'(0)p$ and check that $\ell_0(0.02) - L(0.02) > 10^{-5} > 0$, to conclude $R(p) \geq L(p)$, $0 \leq p \leq 0.02$. We divide the remaining interval $(0.02, 0.25)$ into $6$ intervals: $(0.02, 0.05)$, $(0.05, 0.1)$, $(0.1, 0.15)$, $(0.15, 0.2)$, $(0.2, 0.23)$, $(0.23, 0.25)$, denoted say $(u_i, u_{i+1})$, $i=1,\dots, 6$, choose their midpoints $v_i = \frac{u_i+u_{i+1}}{2}$ and lower-bound $R(p)$ by its tangent $\ell_i(p) = R'(v_i)(p-v_i)+R(v_i)$ and check that $\ell_i(p) > L(p)$ at $p = u_i, u_{i+1}$ (see Table \ref{tab:ell-L2}) to conclude that $R(p) > L(p)$ for all $u_i \leq p \leq u_{i+1}$ $i = 1,\dots, 6$, by convexity.
\end{proof}

\begin{table}[!htb]
\begin{center}
\caption{Proof of Lemma \ref{lm:interp<2}: lower bounds on the differences at the end-points of the linear approximations $\ell_i$ to $R(p)$.}
\label{tab:ell-L2}
\begin{tabular}{r|cccccc}
$i$ & 1 & 2 & 3 & 4 & 5 & 6 \\\hline
$10^4\cdot(\ell_i(u_i)-L(u_i))$             & $.7$ & $1$ & $3$ & $4$ & $5$  & $3$  \\ 
$10^4\cdot(\ell_i(u_{i+1})-L(u_{i+1}))$ & $2$ & $3$ & $4$ & $3$ & $3$ & $2$
\end{tabular}
\end{center}
\end{table}

We are ready to prove the main inequalities of this section.

\begin{proof}[Proof of Lemma \ref{lm:H}]
First we show (a). Lemma \ref{lm:j-U} combined with Lemma \ref{lm:U-G} (ii), (iii) gives (a) for all $0 < p \leq \frac45$, $s \geq 2$, as well as all $0 < p \leq 2$ and $s \geq \frac83$, respectively. It remains to handle the case $\frac45 < p < 2$, $2 \leq s \leq \frac83$. We apply H\"older's inequality, Lemma \ref{lm:interp>2} and {\red \eqref{eq:Hs=2}, equivalently $F(p,2)\ls G(p,2)$,} to get,
\begin{align*}
F(p,s) &\leq F(p, 2)^{\frac{8-3s}{2}}F(p, 8/3)^{\frac{3s-6}{2}} \\
&\leq \Big(G(p,2)\Big)^{\frac{8-3s}{2}}\Big(e^{-p/6}G(p,2)\Big)^{\frac{3s-6}{2}} \\
&= e^{-p\frac{s-2}{4}}2^{p-1}\Gamma(p/2).
\end{align*}
By concavity, $\log s \leq \frac{s-2}{2}+\log 2$, $s \geq 2$, thus
\begin{equation}\label{eq:exp-ps}
e^{-p\frac{s-2}{4}} \leq s^{-p/2}2^{p/2},  \qquad s \geq 2,\ p > 0,
\end{equation}
which gives (a).

To show (b), we proceed similarly. Lemma \ref{lm:j-U} combined with Lemma \ref{lm:U-G} (i) gives (b) for all $0 < p \leq \frac14$ and $s \geq 1.7$. In the remaining case $1.3 \leq s \leq 1.7$, from H\"older's inequality, Lemma \ref{lm:U-G} (i) and Lemma \ref{lm:interp<2}, we obtain
\begin{align*}
F(p,s) &\leq F(p, 1.7)^{\frac{10s-13}{4}} F(p, 1.3)^{\frac{17-10s}{4}}\\
&\leq \Big(G(p, 1.7)\Big)^{\frac{10s-13}{4}}\Big(e^{2p/17}G(p, 1.7)\Big)^{\frac{17-10s}{4}} \\
&= e^{\frac{p}{2}\frac{17-10s}{17}}1.7^{-p/2}2^{3p/2-1}\Gamma(p/2).
\end{align*}
Thanks to concavity, $\log s \leq \frac{10}{17}s-1+\log 1.7$, $s \leq 1.7$, which gives $e^{\frac{p}{2}\frac{17-10s}{17}}1.7^{-p/2} \leq s^{-p/2}$, whence (b).
\end{proof}

\subsection{The integral inequality: $2 < p < 3$}

We follow the general approach from the previous case $p < 2$. Recall \eqref{eq:defF}, $F(p,s) = \int_0^\infty |\jj_1(t)|^st^{p-1}\dd t$ and that the crucial integral inequality \eqref{eq:perfect1} reads $s^{p/2}F(p,s) \leq 2^{p/2}F(p,2).$
Thus here we let
\begin{equation}\label{eq:defHt}
\tilde H(p, s) = s^{-p/2}2^{p/2}F(p,2) - F(p,s), \qquad 2 < p < 3,\  s > 1.
\end{equation}
Note that we can express $F(p,2)$ explicitly: using Corollary \ref{cor:neg-mom-rot} and \eqref{eq:C2p}, we obtain
\begin{align*}
F(p,2) = \int_0^\infty \jj_1(t)^2 t^{p-1} \dd t = \kappa_{p,4}^{-1}\E|\xi_1+\xi_2|^{-p}  &= \kappa_{p,4}^{-1}2^{-p/2}C_{2}(p)\\
&=2^{p-1}\frac{\Gamma\left(\frac{p}{2}\right)\Gamma(3-p)}{\left[\Gamma\left(2-\frac{p}{2}\right)\right]^2\Gamma\left(3-\frac{p}{2}\right)}.
\end{align*}
In view of \eqref{eq:defHt}, we therefore set
\begin{equation}\label{eq:defGt}
\tilde G(p,s) = s^{-p/2}2^{3p/2-1}\Gamma(p/2)D(p)
\end{equation}
with 
\begin{equation}\label{eq:defD}
D(p) = \frac{\Gamma(3-p)}{\left[\Gamma\left(2-\frac{p}{2}\right)\right]^2\Gamma\left(3-\frac{p}{2}\right)},
\end{equation}
so that
\[
\tilde H(p,s) = \tilde G(p,s) - F(p,s).
\]
The main result of this section is that integral inequality \eqref{eq:perfect1} also holds for all $s \geq 2$. We emphasise that $\tilde H(p,2) = 0$.

\begin{lemma}\label{lm:Ht}
The inequality $\tilde H(p,s) > 0$ holds for all $2 < p < 3$ and $s \geq 2$.
\end{lemma}

This will be established in a very much similar way to the previous section: crude pointwise bounds on $\jj_1$ will suffice to handle the case $s \geq \frac{8}{3}$ which will then be extended to $s \geq 2$ by interpolation. 

\begin{lemma}\label{lm:D-log-convex}
With $D(p)$ defined in \eqref{eq:defD}, the function $p \mapsto \log D(p)$  is increasing, convex and positive on $(2,3)$.
\end{lemma}
\begin{proof}
Let $x = \frac{3-p}{2}$, $0 < x < \frac{1}{2}$. By the Legendre duplication formula (see, e.g. 6.1.18 in \cite{AS}),
\[
D(p) = \frac{\Gamma(2x)}{\Gamma(x+\frac12)^2\Gamma(x+\frac32)} = \frac{2^{2x-1}\Gamma(x)}{\sqrt{\pi}\Gamma(x+\frac{1}{2})\Gamma(x+\frac{3}{2})}.
\]
Thus the convexity of $\log D(p)$ on $(2,3)$ is equivalent to the convexity of 
\[
f(x) = \log \Gamma(x) - \log\Gamma\left(x+\frac12\right) - \log\Gamma\left(x+\frac32\right)
\]
on $(0,\frac12)$. Using the series representation of $(\log \Gamma(z))'' = \sum_{n=0}^\infty (z+n)^{-2}$ (see, e.g. 6.4.10 in \cite{AS}), we get
\begin{align*}
f''(x) &= \sum_{n=0}^\infty \frac{1}{(x+n)^2} - \sum_{n=0}^\infty\frac{1}{(x+n+\frac12)^2} - \sum_{n=0}^\infty\frac{1}{(x+n+3/2)^2} \\
&= \frac{1}{x^2} - \frac{1}{(x+\frac12)^2} + \sum_{n=1}^\infty \frac{1}{(x+n)^2} - 2\sum_{n=1}^\infty \frac{1}{(x+n+\frac12)^2}.
\end{align*}
For $0 < x < \frac12$, 
\[
\sum_{n=1}^\infty \frac{1}{(x+n)^2} - 2\sum_{n=1}^\infty \frac{1}{(x+n+\frac12)^2} > \sum_{n=1}^\infty \frac{1}{(\frac12+n)^2} - 2\sum_{n=1}^\infty \frac{1}{(n+\frac12)^2} = -\frac{\pi^2}{2}+4,
\]
thus
\[
f''(x) > \frac{1}{x^2} - \frac{1}{(x+\frac12)^2}  -\frac{\pi^2}{2}+4.
\]
The right hand side is clearly decreasing (e.g., by looking at the derivative), so for $0 < x < \frac12$, it is at least $4 - 1 - \frac{\pi^2}{2} + 4 = 7 - \frac{\pi^2}{2}$ which is positive. 

Moreover, $\frac{\dd}{\dd p}\log D(p)|_{p=2} = \frac{1-\gamma}{2} > 0$ ($\gamma = 0.57..$ is Euler's constant), so $D(p)$ is strictly increasing on $(2,3)$ with $D(2) = 1$.
\end{proof}

\begin{lemma}\label{lm:UGt}
For all $2 < p < 3$ and $s \geq \frac{8}{3}$, we have 
\[
U(p,s) < \tilde G(p,s).
\]
\end{lemma}
\begin{proof}
We let $a = (2\pi)^{1/2}\cdot 15^{1/4}$ and inserting the definitions of $U$ from \eqref{eq:defU} and $\tilde G$ from \eqref{eq:defGt}, the desired inequality becomes
\begin{align*}
\frac{4^pa^{-s}}{3s/2-p} +  s^{-p/2}2^{3p/2-1}&\left(\Gamma\left(p/2\right) - \frac{\Gamma(p/2+2)}{6s} + \frac{\Gamma(p/2+4)}{72s^2}\right) \\
&\qquad\qquad\qquad\qquad < s^{-p/2}2^{3p/2-1}\Gamma\left(\frac{p}{2}\right)D(p),
\end{align*}
equivalently,
\begin{align*}
\frac{2^{p/2+1}a^{-s}}{3s/2-p}s^{p/2+2} < s^2\Gamma\left(\frac{p}{2}\right)(D(p)-1) + \Gamma\left(\frac{p}{2}+2\right)\frac{12s-(p/2+2)(p/2+3)}{72}.
\end{align*}
The right hand side is clearly increasing with $s$ ($D(p) > 1$ by Lemma \ref{lm:D-log-convex}), whereas the left hand side is decreasing with $s$ (for every fixed $2 < p < 3$), as can be checked by examining the derivative of $\log(a^{-s}s^{p/2+2})$. Therefore, it suffices to prove this inequality for $s = \frac{8}{3}$. Moreover, after replacing $\Gamma(\frac{p}{2})$ on the right hand side with $0.88$ (see Lemma \ref{lm:min-gamma}) and $\Gamma(\frac{p}{2}+2)$ with $\Gamma(3) = 2$, it suffices to prove that the function
\[
f(p) = 0.88(8/3)^2(D(p)-1) + \frac{32-(p/2+2)(p/2+3)}{36} - b\frac{\left(16/3\right)^{p/2}}{4-p},
\]
where $b = 2a^{-8/3}(8/3)^2$, is positive for $2 < p < 3$. We put
\[
L(p) = b\frac{\left(16/3\right)^{p/2}}{4-p} + \frac{1}{36}(p/2+2)(p/2+3)
\]
and 
\[
R(p) = 0.88(8/3)^2(D(p)-1) + \frac{8}{9}
\]
which are both convex ($D(p)$ is even log-convex, Lemma \ref{lm:D-log-convex}). For $2 < p < \frac52$, we use the tangent $\ell_1(p) = R(2) + R'(2)(p-2)$ as a lower bound, $R(p) > \ell_1(p)$ and check that at $p = 2$, $p = \frac52$ the linear function $\ell_1$ dominates $L$ (the difference is $0.017..$ and $0.076..$, respectively), which then gives $R > \ell_1 > L$ on $(2,\frac52)$. Similarly, for $\frac52 < p < 3$, $R(p) > \ell_2(p) = R(5/2)+R'(5/2)(p-5/2)$, and $\ell_2 - L$ at $p = \frac52$ and $p = 3$ is $1.19..$ and $3.77..$, respectively. This finishes the proof.
\end{proof}

\begin{lemma}\label{lm:interpol-t}
For all $2 < p < 3$, we have
\[
F(p,8/3) < e^{-p/6}\tilde G(p,2).
\]
\end{lemma}
\begin{proof}
Consider 
\[
L(p) = \log F(p,8/3), \qquad R(p) = \log \left(e^{-p/6}\tilde G(p,2)\right)
\]
which are both convex (recall Lemma \ref{lm:D-log-convex}). Using that, we crudely bound $R(p)$ from below by tangents: $r_1(p) = R(2) + R'(2)(p-2)$ on $(2, 2.5)$ and $r_2(p) = R\left(2.5\right) + R'\left(2.5\right)\left(p-2.5\right)$ on $(2.5, 3)$ and then compare their values at the end points with upper bounds on $L$ to conclude that $r_1 > L$ on $(2, 2.5)$ and $r_2 > L$ on $(2.5, 3)$. Estimates \eqref{eq:83-1} and \eqref{eq:83-2} added together (applied with $m=100$ as in Lemma \ref{lm:interp>2}) yield
\[
L(2) < 0.35, \quad L(2.5) < 0.56, \quad L(3) < 0.96,
\]
whereas we check directly that
\[
r_1(2) > 0.359, \quad r_1(2.5) > 0.58, \quad r_2(3) > 1.48.
\]
Comparing these values finish the argument.
\end{proof}

\begin{proof}[Proof of Lemma \ref{lm:Ht}]
{\red Lemma \ref{lm:j-U} combined with Lemma \ref{lm:UGt} show that $\tilde{H}(p,s)>0$ for all $2<p<3$ and $s\gr 8/3$. To cover the regime $2\ls s< 8/3$, we first apply H\"{o}lder's inequality in the exact same way as in the proof of Lemma \ref{lm:H} (a),
\[
F(p,s)\ls F(p,2)^{\frac{8-3s}{2}}F(p,8/3)^{\frac{3s-6}{2}} 
\]
and now, with $F(p,2) = \tilde G(p,2)$ and Lemma \ref{lm:interpol-t}, we get that
\[
F(p,s) \ls e^{-p\frac{s-2}{4}}\tilde{G}(p,2).
\]
Finally, using \eqref{eq:exp-ps}, the right hand side gets upper bounded by the desired $\tilde G(p,s)$.}
\end{proof}

\subsection{Miscellaneous facts}

Our first result here is a straightforward extension of Lemma 8 from \cite{KK} to negative moments (see also Lemma 3 in \cite{CKT}). 

\begin{lemma}\label{lm:vec-to-1dim}
Let $0 < p < 1$. Let $n, d \geq 1$ and let $X_1, \dots, X_n$ be independent rotationally invariant random vectors in $\R^d$. Then
\[
\E\left|\sum_{k=1}^n {\red |v_k|}X_k\right|^{-p} = \beta_{p,d}\E\left|\sum_{k=1}^n \scal{v_k}{X_k}\right|^{-p} 
\]
for arbitrary vectors $v_1, \dots, v_n$ in $\R^d$, where
\[
\beta_{p,d} = \frac{\sqrt{\pi}\Gamma\left(\frac{d-p}{2}\right)}{\Gamma\left(\frac{1-p}{2}\right)\Gamma\left(\frac{d}{2}\right)}.
\]
\end{lemma}
\begin{proof}
Thanks to homogeneity, we can assume that the $v_k$ are \emph{unit}. Thanks to rotational invariance and independence, we can assume without loss of generality that $v_1 = \dots = v_n = e_1$, but then it suffices to consider the case $n=1$ (because sums of independent rotationally invariant random vectors are rotationally invariant). The latter can be easily justified in a number of ways. 

For instance, it follows from a Fourier-analytic argument: we invoke \eqref{eq:neg-mom-rot}, rewrite $\E\ \jj_{d/2-1}(t|X_k|)$ as $\E e^{it\scal{v_k}{X_k}}$ and apply \eqref{eq:formula-mom} with $d=1$ to $\sum \scal{v_k}{X_k}$ which gives $\beta_{p,d} = \kappa_{p,d}/(2K_{p,1})$. 

Alternatively, we can apply a standard embedding-type argument: if we take a random vector $\xi$ uniform on the unit Euclidean sphere $S^{d-1}$, independent of the $X_k$, we have for every vector $x$ in $\R^d$
\[
\E|\scal{x}{\xi}|^{-p} = \beta_{p,d}^{-1}|x|^{-p}
\]
with
\[
\beta_{p,d}^{-1} = \E|\scal{e_1}{\xi}|^{-p} = \frac{\int_{-1}^1 |t|^{-p}(1-t^2)^{\frac{d-3}{2}} \dd t}{\int_{-1}^1 (1-t^2)^{\frac{d-3}{2}} \dd t} = \frac{\Gamma\left(\frac{1-p}{2}\right)\Gamma\left(\frac{d}{2}\right)}{\sqrt{\pi}\Gamma\left(\frac{d-p}{2}\right)}.
\]
Applying this to $x = X_1$, taking the expectation over $X_1$ and noting that $\scal{X_1}{\xi}$ has the same distribution as $\scal{X_1}{e_1}$ finishes the argument.
\end{proof}

\begin{lemma}\label{lm:Gamma-small-coeff}
For every $0 < q < 2$, we have
\[
 \left(\frac{13}{20}\right)^{q} < \Gamma(2-q).
\]
\end{lemma}
\begin{proof}
The function $f(q) = \log\Gamma(2-q) - q\log \frac{13}{20}$ is convex  on $(0,2)$ with $f'(0) = \gamma - 1 -\log\frac{13}{20} > 0.007$. Thus $f$ is strictly increasing and the lemma follows since $f(0) = 0$.
\end{proof}

\section{End of the proof of Theorem \ref{thm:OPneg-p<2}}\label{sec:end-proof-p<2}

To finish the proof of Theorem \ref{thm:OPneg-p<2}, we only need to justify Steps 1-4 from Section \ref{sec:beginning-proof-p<2}.

\subsection{Step 1 and 3: Integral inequality}
Lemma \ref{lm:H} (a) and (b) gives Step 1 and 3, respectively.

\subsection{Step 2: Induction}\label{sec:ind}
First note that, by homogeneity, \eqref{eq:OPneg} with $C(p) = C_\infty(p)$ is equivalent to
\[
\E\left|\xi_1 + \sum_{k=2}^n a_k\xi_k\right|^{-p} \leq C_\infty(p)\left(1 + \sum_{k=2}^n a_k^2\right)^{-p/2}.
\]
For $p > 0$ and $x \geq 0$ we define
\[
\phi_p(x)=(1+x)^{-p/2}
\]
and
\[
\Phi_p(x)=\begin{cases} \phi_p(x), & x\geq 1,\\
                        2\phi_p(1)-\phi_p(2-x), & 0\leq x\leq 1.
\end{cases}
\]
Geometrically, on $[0,1]$, the graph of $\Phi_p(x)$ is obtained from the graph of $\phi_p(x)$ on $[1,2]$ by reflecting it about $(1, \phi_p(1))$. Crucially, $\Phi_p(x) \leq \phi_p(x)$ for all $x \geq 0$, {\red since $2\phi_p(1)\ls \phi_p(x)+\phi_p(2-x)$, by the convexity of $\phi_p$}. By induction on $n$, we will show a strengthened version of the above with $\phi_p$ on the right hand side replaced by $\Phi_p$.

\begin{theorem}\label{thm:ind}
Let $\frac14 \leq p \leq 2$. Let $\xi_1, \xi_2, \dots$ be independent random vectors uniform on the unit Euclidean sphere $S^{3}$ in $\R^4$.  For every $n \geq 2$ and nonnegative numbers $a_2, \dots, a_n$, we have
\begin{equation}\label{eq:ind}
\E\left|\xi_1+\sum_{k=2}^n a_k\xi_k\right|^{-p} \leq C_\infty(p)\Phi_p\left(\sum_{k=2}^n a_k^2\right).
\end{equation}
\end{theorem}
\begin{proof}
For the inductive base, when $n=2$, \eqref{eq:ind} becomes
\[
\E|\xi_1 + \sqrt{t}\xi_2|^{-p} \leq 2^{p/2}\Gamma\left(2- \frac{p}{2}\right)\Phi_p(t), \qquad t \geq 0,
\]
where we have put $t = a_2^2$. By homogeneity and the fact that $\Phi_p \leq \phi_p$, the case $t \geq 1$ reduces to the case $0 \leq t \leq 1$. {\red Indeed, if $t \geq 1$, $\Phi_p(t) = \phi_p(t) = (1+t)^{-p/2}$, so dividing both sides by $t^{-p/2}$, the inequality is equivalent to the one with $1/t$ instead of $t$ and $\phi_p(1/t)$ on the right hand side}. The case $0 \leq t \leq 1$ follows by combining Corollary \ref{cor:two-coeff-R4} and Lemma \ref{lm:ind-base} (applied to $q = p/2$, noting as usual that by rotational invariance, $\E|e_1 + \sqrt{t}\xi_2|^{-p} = \E|\xi_1+\sqrt{t}\xi_2|^{-p}$).

For the inductive step, let $n \geq 2$ and suppose \eqref{eq:ind} holds for all $n-1$ nonnegative numbers $a_2, \dots, a_n$. To prove it for $n$ nonnegative arbitrary numbers, say $a_2, \dots, a_n, a_{n+1}$, we let 
\[
x = a_2^2+\dots+a_n^2+a_{n+1}^2
\]
and consider 3 cases.

\emph{Case 1: $a_k > 1$ for some $2 \leq k \leq n+1$.} Then $x > 1$, so $\Phi_p(x) = \phi_p(x)$ and our goal is to show
\begin{equation}\label{eq:ind1}
\E\left|\sum_{k=1}^{n+1} a_k\xi_k\right|^{-p} \leq C_\infty(p)\left(\sum_{k=1}^{n+1} a_k^2\right)^{-p/2}
\end{equation}
where we put $a_1 = 1$. Let $a_1^*, \dots, a_{n+1}^*$ be a nonincreasing rearrangement of the sequence $a_1, \dots, a_{n+1}$ and set $a_k' = \frac{a_k^*}{a_1^*}$, $k = 1, \dots, n+1$. Thanks to homogeneity, to prove \eqref{eq:ind1}, it is enough to prove
\[
\E\left|\sum_{k=1}^{n+1} a_k'\xi_k\right|^{-p} \leq C_\infty(p)\Phi_p\left(\sum_{k=2}^{n+1} a_k'^2\right)
\]
which is handled by either of the next two cases because here $a_1'=1$ and $a_k' \leq 1$ for all $k \geq 2$.

\emph{Case 2.1: $a_k \leq 1$ for all $2 \leq k \leq n+1$ and $x \geq 1$.} Since $x \geq 1$, our goal is again \eqref{eq:ind1} with $a_1 = 1$. We have,
\[
\max_{k \leq n+1} a_k = 1 \leq \frac{1}{\sqrt{2}}\sqrt{1+x} = \frac{1}{\sqrt{2}}\left(\sum_{k=1}^{n+1} a_k^2\right)^{1/2},
\]
so Corollary \ref{cor:all-small-p<2} finishes the inductive argument in this case.

\emph{Case 2.2: $a_k \leq 1$ for all $2 \leq k \leq n+1$ and $x < 1$.} 
Fix vectors $v_2, \dots, v_{n+1}$ in $\R^4$ with $|v_k| = a_k$, for each $k = 2, \dots, n+1$. Then, plainly,
\[
\E\left|\xi_1+\sum_{k=2}^{n+1}a_k\xi_k\right|^{-p} =\E\left||e_1|\xi_1+\sum_{k=2}^{n+1}|v_k|\xi_k\right|^{-p}
\]
and thanks to Lemma \ref{lm:vec-to-1dim}, when $0 < p < 1$, the right hand side can be written as
\[
\E\left||e_1|\xi_1+\sum_{k=2}^{n+1}|v_k|\xi_k\right|^{-p} = \beta_{p,4}\E\left|\scal{e_1}{\xi_1} +\sum_{k=2}^{n+1}\scal{v_k}{\xi_k}\right|^{-p}.
\]
If we let $Q$ be a random orthogonal matrix, independent of the $\xi_k$ and note that $(\xi_n, \xi_{n+1})$ has the same distribution as $(\xi_n, Q\xi_n)$, we obtain
\[
\E\left|\scal{e_1}{\xi_1} +\sum_{k=2}^{n+1}\scal{v_k}{\xi_k}\right|^{-p} = \E_Q\E_\xi\left|\scal{e_1}{\xi_1} +\sum_{k=2}^{n-1}\scal{v_k}{\xi_k} + \scal{v_n+Q^\top v_{n+1}}{\xi_n}\right|^{-p}\hspace*{-0.7em}.
\]
Going back to the vector sum again via Lemma \ref{lm:vec-to-1dim}, we arrive at the identity
\[
\E\left|\xi_1+\sum_{k=2}^{n+1}a_k\xi_k\right|^{-p} = \E_Q\E_\xi\left|\xi_1+\sum_{k=2}^{n-1}|v_k|\xi_k + |v_n+Q^\top v_{n+1}|\xi_n \right|^{-p}.
\] 
The same identity continues to hold for all $0 < p < 3$: we know it holds for all $0 < p < 1$ and both sides are clearly analytic in $p$ wherever the expectations exists, so in $\{p \in \C, \ \textrm{Re}(p) <3\}$, {\red because $|\E|\cdot|^{z}| \leq \E|\cdot|^{\textrm{Re}(z)}$ for $z \in \C$}  (the analyticity follows, e.g. from Morera's theorem by a standard argument).  Conditioned on the value of $Q$, the inductive hypothesis applied to the $n-1$ nonnegative numbers $|v_2|, \dots, |v_{n-1}|, |v_n+Q^\top v_{n+1}|$ yields
\[
\E\left|\xi_1+\sum_{k=2}^{n+1}a_k\xi_k\right|^{-p} \leq \E_Q C_\infty(p)\Phi_p\left(|v_2|^2 + \dots + |v_{n-1}|^2 + |v_n+Q^\top v_{n+1}|^2\right).
\]
Note that
\[
|v_2|^2 + \dots + |v_{n-1}|^2 + |v_n \pm Q^\top v_{n+1}|^2 = x \pm 2\scal{v_n}{Q^\top v_{n+1}},
\]
so thanks to the symmetry of the distribution of $Q$, we can rewrite the right hand side as
\[
C_\infty(p)\E_Q \frac{\Phi_p(x+2\scal{v_n}{Q^\top v_{n+1}}) + \Phi_p(x-2\scal{v_n}{Q^\top v_{n+1}}) }{2}.
\]
The proof of the inductive step now follows from the following \emph{extended} concavity property of $\Phi_p$ applied to $a_\pm = x \pm 2\scal{v_n}{Q^\top v_{n+1}}$.
\end{proof}

\begin{lemma}
Let $p > 0$. For every $a_-, a_+ \geq 0$ with $\frac{a_- + a_+}{2} \leq 1$, we have
\[
\frac{\Phi_p(a_-) + \Phi_p(a_+)}{2} \leq \Phi_p\left(\frac{a_- + a_+}{2}\right).
\]
\end{lemma}
\begin{proof}
This is Lemma 20 in \cite{CKT} (stated there for \emph{no} reason only for $0 < p < 1$, as the proof works for every $p > 0$ because it only uses the convexity of $\phi_p$).
\end{proof}

\subsection{Step 4: Projection}
Let us say that $a_1 = \max_{k \leq n} |a_k|$, so $a_1 > \sqrt{\frac{10}{13}}$. \emph{Projecting} onto this coefficient, that is applying Corollary \ref{cor:two-coeff-min} to $a = a_1$ and $v = \sum_{k=2}^n a_k\xi_k$ (conditioning on its value), we get
\[
\E\left|\sum_{k=1}^{n}a_k\xi_k\right|^{-p} \leq a_1^{-p} \leq \left(\frac{13}{10}\right)^{p/2} \leq 2^{p/2}\Gamma\left(2-\frac{p}{2}\right) = C_\infty(p),
\]
where the last inequality results from Lemma \ref{lm:Gamma-small-coeff} (applied to $q = p/2$). This finishes the proof of Theorem \ref{thm:OPneg-p<2}.\hfill$\square$

\section{End of the proof of Theorem \ref{thm:OPneg-p>2}}\label{sec:end-proof-p>2}

To finish the proof of Theorem \ref{thm:OPneg-p>2}, we only need to show here Steps 1 and 2 from Section \ref{sec:beginning-proof-p>2}.

\subsection{Step 1: Integral inequality}
Lemma \ref{lm:Ht} gives the desired claim.

\subsection{Step 2: Induction}
We repeat the entire inductive argument from Section \ref{sec:ind} verbatim, replacing $\frac14 \leq p \leq 2$ with $2 < p  < 3$ and $C_\infty(p)$ with $C_2(p)$. The only modification required  is to check the inductive base which now amounts to verifying that
\[
\E|\xi_1 + \sqrt{t}\xi_2|^{-p} \leq C_2(p)\Phi_p(t) = C_2(p)(2^{1-p/2}-(3-t)^{-p/2}), \qquad 0 \leq t \leq 1.
\]
By Lemma \ref{lm:two-coeff-expansion}, the left hand side is clearly increasing in $t$ (when $2 < p < 3$ and $d=4$ all the coefficients in the power series expansion therein are positive), whereas the right hand side is clearly decreasing in $t$. By the definition of $C_2(p)$, there is equality at $t=1$. This finishes the  whole proof.

\appendix

\section*{Appendix: Behaviour of the constants}\label{app:pht}

We sketch an argument of the following proposition which justifies \eqref{eq:phase.transition}.
\begin{proposition}
For every $d\gr 1$, the equation $c_{d,2}(q)=c_{d,\infty}(q)$ has a unique solution $q = q_d^\ast$ in $(-(d-1),2)$. Moreover, $c_{d,2}(q) < c_{d,\infty}(q)$ for $-(d-1) < q < q_d^*$ and $c_{d,2}(q) > c_{d,\infty}(q)$ for $q_d^* < q < 2$. For $d \geq 5$, we have $q_d^* \in (-(d-1), -(d-2))$.
\end{proposition}

\begin{proof}
Since the cases $1\ls d\ls 4$ have been explicitly dealt with (see the discussion in the introduction), it is enough to analyse the case $d\gr 5$. Moreover, by the Schur-concavity result of \cite{BC} and \cite{KK}, $c_{d,\infty}(q) < c_{d,2}(q)$ for every $0 < q < 2$, so we can further assume that $-(d-1) < q < 0$. We look into the sign of
\[
h_d(q) = \log(c_{d,2}(q)^q)-\log(c_{d,\infty}(q)^q).
\]
{\red Note that for $q<0$ the sign of $h_d(q)$ is opposite to the sign of $c_{d,2}(q)-c_{d,\infty}(q)$. Now, $h_d(q)$} can be equivalently recast as
\begin{align*}
h_d(q) &= \log\left(2^{-\frac{q}{2}}\frac{\Gamma\left(\frac{d}{2}\right)\Gamma(d+q-1)}{\Gamma\left(\frac{d+q}{2}\right)\Gamma\left(d+\frac{q}{2}-1\right)}\right)-\log\left(\left(\frac{2}{d}\right)^{q/2}\frac{\Gamma\left(\frac{d+q}{2}\right)}{\Gamma\left(\frac{d}{2}\right)}\right)\\
       &= -q\log 2+\frac{q}{2}\log d + \log\left(\frac{\Gamma\left(\frac{d}{2}\right)^2\Gamma(d+q-1)}{\Gamma\left(\frac{d+q}{2}\right)^2\Gamma\left(d+\frac{q}{2}-1\right)}\right).
\end{align*}
Writing $x=\frac{q+d-1}{2}\in\left(0,\frac{d-1}{2}\right)$ and $\tilde{h}_d(x)=h_d(2x+1-d)$, we get (using the Legendre duplication formula $\Gamma(2x)\sqrt{\pi}=2^{2x-1}\Gamma(x)\Gamma(x+1/2)$) that
\[
\tilde{h}_d(x) = x\log d + \log\left(\frac{\Gamma(x)}{\Gamma\left(x+\frac{1}{2}\right)\Gamma\left(x+\frac{d-1}{2}\right)}\right)+\log\left(\frac{2^{d-2}\Gamma\left(\frac{d}{2}\right)^2}{\sqrt{\pi}d^{\frac{d-1}{2}}}\right).
\]
We now make the following claims.

\noindent\textbf{Claim 1.} For all $0 < x < 1$, $\tilde{h}_d''(x) > 0.06$.

\noindent\textbf{Claim 2.} For every $d \geq 5$, $\inf_{1< x < \frac{d-1}{2}} \tilde{h}_d'(x) > 0$.

\noindent\textbf{Claim 3.} $\tilde{h}_d(\frac12) < 0$.

The strict convexity from Claim 1, the simple observation that $\tilde{h}_d(0+) = +\infty$ and Claim 3 give that that $\tilde{h}_d$ has a unique zero, say $x_0$ in $(0,\frac12)$, is positive on $(0,x_0)$ and negative on $(x_0, \frac12)$. Claim 2 and the simple observation that $\tilde{h}_d(\frac{d-1}{2}) = 0$ gives that $\tilde{h}_d$ is negative on $\left[1,\frac{d-1}{2}\right)$. Convexity also gives that $\tilde{h}_d$ is negative on $(\frac12, 1)$, for $h_d(\frac12)$ and $h_d(1)$ are negative. These give the desired behaviour of $c_{d,2}(q) - c_{d,\infty}(q)$ for $-(d-1) < q < 0$. Finally, it also follows from Claim 2 that $h_d'(0) > 0$ which gives $c_{d,2}(0) - c_{d,\infty}(0) > 0$. It remains to prove the claims. 
\end{proof}

\begin{proof}[Proof of Claim 1]
Differentiating twice yields
\[
\tilde{h}_d''(x) = \sum_{n=0}^\infty\left(\frac{1}{(x+n)^2}-\frac{1}{\left(x+n+\frac{1}{2}\right)^2}-\frac{1}{\left(x+n+\frac{d-1}{2}\right)^2}\right).
\]
Note that the first two terms make up a decreasing function, thus for $0 < x < 1$ and $d \geq 5$ the right hand side is greater than
\[
\sum_{n=0}^\infty\left(\frac{1}{(1+n)^2}-\frac{1}{\left(1+n+\frac{1}{2}\right)^2}-\frac{1}{\left(n+2\right)^2}\right) = 5-\frac{\pi^2}{2}>0,
\]
which proves the claim.
\end{proof}

\begin{proof}[Proof of Claim 2]
Differentiating once yields
\[
\tilde{h}_d'(x) = \log d + \left(\psi(x) - \psi\left(x+\frac12\right)\right) - \psi\left(x + \frac{d-1}{2}\right)
\]
where $\psi = (\log \Gamma)'$ as usual denotes the digamma function. By the well-known inequality $\psi(u) \leq \log u -\frac{1}{2u}$, $u > 0$ (see, e.g. 6.3.21 in \cite{AS}), we obtain
\[
\tilde{h}_d'(x) \geq \log d - \log\left(x + \frac{d-1}{2}\right) + \frac{1}{2x+d-1} - \left(\psi\left(x+\frac12\right)-\psi(x)\right).
\]
Put $y = \frac{d-1}{2}$ and call the right hand side $F(x,y)$. Note that for every fixed $x > 1$,
\[
\frac{\partial F}{\partial y}(x,y) = \frac{1}{y+1/2} - \frac{1}{x+y} - \frac{1}{2}\frac{1}{(x+y)^2} > \frac{1}{y+1/2} - \frac{1}{1+y} - \frac{1}{2}\frac{1}{(1+y)^2}
\]
which is clearly positive for all $y > 0$. Therefore, for all $1 < x < y$,
\[
\tilde{h}_d'(x) \geq F(x,y) > F(x,x).
\]
It remains to prove that $f(x) = F(x,x) > 0$ for every $x > 1$. We have,
\[
f(x) = \left(\log\left(1+\frac{1}{2x}\right) + \frac{1}{4x}\right) - \left(\psi\left(x+\frac12\right)-\psi(x)\right).
\]
Note that each bracket is a decreasing function in $x$ (for the second one, e.g. by taking the derivative). Thus, crudely, for $1 < x < 1.07$,
\[
f(x) > \left(\log\left(1+\frac{1}{2\cdot 1.07}\right) + \frac{1}{4\cdot 1.07}\right) -  \left({\red \psi\left(1+\frac12\right)-\psi(1)}\right) > 0.003.
\]
For $x \geq 1.07$, using again $\psi(x+1/2) \leq \log (x+1/2) - \frac{1}{2x+1}$ as well as $\psi(x) \geq \log(x+1/2) - \frac{1}{x}$ (see \cite{EGP}), we get
\[
f(x) \geq \log\left(1+\frac{1}{2x}\right) + \frac{1}{4x} - \left(\frac{1}{x}-\frac{1}{2x+1}\right)
\]
It is elementary to verify that the right hand side is positive for $x \geq 1.07$ (it is in fact unimodal, e.g. by analysing its derivative).
\end{proof}

\begin{proof}[Proof of Claim 3]
We have, $\tilde{h}_d\left(\frac{1}{2}\right) = \log\left(2^{d-2}d^{1-d/2}\Gamma\left(\frac{d}{2}\right)\right)$. Letting $u = d/2 \geq 5/2$ and using {\red 
\begin{equation}\label{eq:Stirling}
\Gamma(u)\ls \sqrt{2\pi}u^{u-\frac{1}{2}}e^{-u+\frac{1}{12u}}, \qquad u > 0,
\end{equation} 
 (see \cite{J}),} we get 
\[
\tilde{h}_d\left(\frac{1}{2}\right)  < \log\left(\sqrt{2\pi}2^{u-1}e^{-u+\frac{1}{12u}}u^{1/2}\right) \leq  \log\left(\sqrt{2\pi}2^{u-1}e^{-u+\frac{1}{30}}u^{1/2}\right).
\]
Denoting the right hand side by $f(u)$, we see that $f$ is strictly concave. Since $f'(5/2) < -0.1$, $f$ is decreasing for $u \geq 5/2$. Thus $f(5/2) < -0.04$ finishes the argument.
\end{proof}

\begin{remark}\label{rem:q*-d-large}
We have,
\begin{equation}\label{eq:q*-d-large}
q_d^* = -(d-1) + O(d)\text{exp}\left(-\frac{1-\log 2}{2}d\right), \qquad d \to \infty.
\end{equation}
As before, by Claim 1, to show $q_d^* < -(d-1)+2\alpha_d$ for some $\alpha_d > 0$, it suffices to check that $\tilde{h}_d(\alpha_d) < 0$. {\red We have,
\begin{align*}
\tilde h_d(\alpha_d) &= \alpha_d\log d + \log\frac{\Gamma(\alpha_d)}{\Gamma(\alpha_d+\frac12)\Gamma(\alpha_d + \frac{d-1}{2})} + \log \frac{2^{d-2}\Gamma(\frac{d}{2})^2}{\sqrt{\pi}d^{\frac{d-1}{2}}} \\
&= \alpha_d\log d + \log\frac{\Gamma(\alpha_d)\Gamma(\frac{d}{2})^2}{\Gamma(\alpha_d+\frac12)\Gamma(\alpha_d + \frac{d-1}{2})d^{\frac{d-1}{2}}} + d\log 2 - \log(4\sqrt{\pi}).
\end{align*}
Note that $\Gamma(x) < \frac{1}{x}$, for $0 < x < 1$ (since $\Gamma(1+x) < 1$).
We consider $\alpha_d = Cde^{-cd}$ for positive constants $c, C$ chosen soon. For $d$ large enough, $\alpha_d < \frac{1}{2}$, so $\Gamma(\alpha_d+\frac12) > 1$. Moreover, $\Gamma(\alpha_d) < \frac{1}{\alpha_d}$, $\Gamma(\alpha_d+\frac{d-1}{2}) \geq \Gamma(\frac{d-1}{2}) = \frac{2}{d-1}\Gamma(\frac{d+1}{2}) > \frac{2}{d}\Gamma(\frac{d}{2})$, as well as $\alpha_d\log d = o(1)$, thus
\begin{align*}
\tilde h_d(\alpha_d) &\leq O(1) + \log\frac{d\cdot \Gamma(\frac{d}{2})}{\alpha_dd^{\frac{d-1}{2}}} + d\log 2  \\
&= O(1) + \log\frac{\Gamma(\frac{d}{2})}{Ce^{-cd}d^{\frac{d-1}{2}}} + d\log 2.
\end{align*}
Applying \eqref{eq:Stirling} to $\Gamma(\frac{d}{2})$, we obtain
\begin{align*}
\tilde h_d(\alpha_d) &\leq O(1) - \log C + d\left(c+\frac{1}{2}\log 2-\frac{1}{2}\right).
\end{align*}
Choosing $c = \frac{1-\log 2}{2}$ and $C$ large enough to offset $O(1)$, the right hand side becomes negative.}
\end{remark}

\end{document}